\documentclass{amsart}
\usepackage{graphicx} % Required for inserting images
\usepackage{pinlabel} %for labels of figures
\usepackage{amssymb} % For \Bbbk and other math symbols
\usepackage{mathrsfs} % For \mathscr command 
\usepackage[all]{xy}
\usepackage{geometry}
\usepackage{enumerate}

\usepackage{aliascnt} %To make autoref work as expected
\usepackage[colorlinks=true,linkcolor=blue,citecolor=red,urlcolor=red,citebordercolor={0 0 1},urlbordercolor={0 0 1},linkbordercolor={0 0 1}]{hyperref} %needs to be loaded after most things
\usepackage[nameinlink]{cleveref}

\renewcommand{\bar}{\overline}
\renewcommand{\tilde}{\widetilde}
\newcommand{\tensor}{\otimes}
\def\gitq{/\!/}
\newcommand{\base}{\Bbbk}
\newcommand{\minus}{\smallsetminus}
\newcommand{\sS}{{\textsf{S}}}
\newcommand{\into}{\hookrightarrow}
\newcommand{\co}{\colon}

\newcommand{\bA}{\mathbb{A}}
\newcommand{\bC}{\mathbb{C}}
\newcommand{\bR}{\mathbb{R}}
\newcommand{\bG}{\mathbb{G}}
\newcommand{\bP}{\mathbb{P}}
\newcommand{\CP}{\mathbb{C}\mathbb{P}}
\newcommand{\bZ}{\mathbb{Z}}
\newcommand{\cC}{\mathcal{C}}

\newcommand{\cK}{\mathcal{K}}
\newcommand{\cL}{\mathcal{L}}
\newcommand{\cM}{\mathcal{M}}
\newcommand{\oh}{\mathcal{O}}
\newcommand{\cS}{\mathcal{S}}
\newcommand{\cU}{\mathcal{U}}
\newcommand{\cV}{\mathcal{V}}
\newcommand{\cW}{\mathcal{W}}
\newcommand{\cX}{\mathcal{X}}

\newcommand{\bQ}{\mathbb{Q}}
\newcommand{\bS}{\mathbb{S}}
\newcommand{\cG}{\mathcal{G}}
\newcommand{\cB}{\mathcal{B}}

\newcommand{\Spec}{\operatorname{Spec}}
\newcommand{\Proj}{\operatorname{Proj}}

\newcommand{\Hom}{\operatorname{Hom}}
\newcommand{\Pic}{\operatorname{Pic}}
\newcommand{\Aut}{\operatorname{Aut}}
\newcommand{\Gr}{\operatorname{Gr}}
\newcommand{\QCoh}{\operatorname{QCoh}}

\newcommand{\SL}{\operatorname{SL}}
\newcommand{\GL}{\operatorname{GL}}

\newcommand{\Sym}{\operatorname{Sym}}
\newcommand{\rk}{\operatorname{rk}}
\newcommand{\ps}{\operatorname{ps}}
\renewcommand{\ss}{\operatorname{ss}}
\newcommand{\pre}{\operatorname{pre}}
\newcommand{\onto}{\twoheadrightarrow}
\newcommand{\mapsonto}{\twoheadrightarrow}

\newcommand{\cartesian}{\ar@{}[ul]|{\square}}
\newcommand{\can}{\hspace{.05cm}\operatorname{can}}
\newcommand{\CM}{\operatorname{CM}}
\newcommand{\xto}{\xrightarrow}

\newcommand{\Hilb}{\operatorname{Hilb}}

\newcommand{\Quot}{\operatorname{Quot}}
\newcommand{\Chow}{\operatorname{Chow}}
\newcommand{\Ch}{\operatorname{Ch}}

\newcommand{\cCoh}{\mathcal{C}\mathit{oh}}
\newcommand{\cBun}{\mathcal{B}\mathit{un}} 
\newcommand{\Fanos}{\mathcal{F}\hspace{-.08cm}\mathit{anos}}

\DeclareMathOperator{\Grad}{Grad}
\DeclareMathOperator{\Filt}{Filt}
\DeclareMathOperator{\Map}{Map}
\DeclareMathOperator{\wt}{wt}
\DeclareMathOperator{\ev}{ev}
\DeclareMathOperator{\univ}{univ}
\DeclareMathOperator{\ADeg}{ADeg}
\DeclareMathOperator{\ST}{\overline{ST}}

\newcommand{\Deg}{\mathscr{D}\hspace{-.04cm}\mathit{eg}}
\newcommand{\Comp}{\mathscr{C}\hspace{-.04cm}\mathit{omp}}

\newcommand{\all}{\text{all}}
\newcommand{\irr}{\text{irr}}

\newcommand{\NE}{\text{NE}}
\newcommand{\st}{\text{st}}

\renewcommand{\H}{\mathrm{H}}
\newcommand{\R}{\mathrm{R}}

\theoremstyle{plain}
\newtheorem{theorem}[equation]{Theorem}
\newtheorem{corollary}[equation]{Corollary}
\newtheorem{lemma}[equation]{Lemma}
\newtheorem{proposition}[equation]{Proposition}

\theoremstyle{definition}
\newtheorem{definition}[equation]{Definition}
\newtheorem{example}[equation]{Example}

\newtheorem{problem}[equation]{Problem}
\newtheorem{question}[equation]{Question}
\newtheorem{refined-question}[equation]{Refined Question}
\newtheorem{conjecture}[equation]{Conjecture}

\theoremstyle{remark}
\newtheorem{remark}[equation]{Remark}

\numberwithin{equation}{section}

\renewcommand{\tag}[1]{\cite[Tag\,\href{http://stacks.math.columbia.edu/tag/#1}{#1}]{stacks-project}}

\usepackage[utf8]{inputenc} % usually not needed (loaded by default)
\usepackage[T1]{fontenc}

\usepackage[nameinlink]{cleveref}

\begin{document}

\title{The intrinsic approach to moduli theory}
\begin{abstract}
Moduli theory has captured the imagination of algebraic geometers for at least two centuries.  Up
until the end of the 20th century, moduli spaces were constructed and studied by rigidifying the
moduli problem using extrinsic data and applying geometric invariant theory. Over the last several
decades, there has been a paradigm shift toward studying moduli problems intrinsically using the
language of algebraic stacks.  We highlight recent advances in this direction that have incorporated
ideas from geometric invariant theory to develop a structure theory for algebraic stacks. In the
ideal situation, it allows one to decompose an algebraic stack into simpler strata and construct
moduli spaces corresponding to each stratum. In addition to surveying some previous applications of
the theory, we take a forward-looking perspective on the field and identify questions for future
research.
\end{abstract}

\author{Jarod Alper}
\address{
  Department of Mathematics \\ 
  University of Washington \\
  Seattle, Washington, USA}
\email{jarod@uw.edu}

\author{Daniel Halpern-Leistner}
\address{Cornell University, Ithaca, New York, USA}
\email{daniel.hl@cornell.edu}

\maketitle
\setcounter{tocdepth}{1}
\tableofcontents

\section{Introduction}

There is something enchanting about moduli spaces. They simultaneously provide an answer to one of
the most fundamental questions in mathematics, namely the \emph{classification problem}, and are
fascinating spaces in their own right.  These spaces are often projective varieties with a very
rich geometry, reflecting the geometry of the individual objects that they parameterize. It has
been nearly 200 years since Riemann coined the term `moduli' and introduced the moduli spaces of
curves \cite{riemann}. Similar to the profound paradigm shift around 100 years ago from studying
varieties as embedded spaces to abstract spaces, there has been a gradual transition beginning
around 30 years ago from studying moduli spaces as quotients of rigidified parameter spaces (e.g.,
space of all embedded curves) to a more intrinsic approach using algebraic stacks.  While algebraic
stacks were first introduced  in the 1960s \cite{deligne-mumford}, it was not until the 1990s that
they had significant impact in moduli theory. This adoption was spurred by key results such as the
Keel--Mori Theorem \cite{keel-mori} (which for instance implies the existence of $\bar{M}_g$ as a
proper algebraic space) and Kollár's Ampleness Criterion \cite{kollar-projectivity} (which together
with birational methods implies the projectivity of $\bar{M}_g$).  It still took another decade or two
until they entered mainstream algebraic geometry. 

The past decade has seen a new wave of intrinsic methods alongside the discovery of new moduli
spaces. We outline several major advances in the structure theory of algebraic stacks including the local
quotient structure (\S \ref{sec:local-structure}), an intrinsic approach to construct and study good
moduli spaces and $\Theta$-stratifications called `Beyond GIT' (\S \ref{sec:beyond-git}), and combinatorial structures attached
to an algebraic stack, e.g., degeneration spaces (\S \ref{sec:combinatorial-structures}).  We ground these techniques by discussing two specific moduli problems of a very different nature.  In \S \ref{sec:beyond-git}, we use the relatively easy and `linear' moduli problem of semistable vector bundles on a curve to show how each of several approaches in beyond GIT can be effectively applied to construct and study the moduli space. In \S \ref{sec:singular-curves}, we take a quick detour to explore the geometry of $\bar{M}_g$, which has the wonderful feature that many of its birational models are moduli spaces themselves, albeit of different classes of singular curves.  These birational models were recently constructed using a blend of intrinsic and extrinsic methods, 
but challenges remain in extending the beyond GIT techniques to further study these `non-linear' moduli problems of singular varieties.  We take a forward-looking perspective on the field by raising interesting conjectures and problems
prime for future research (\S \ref{sec:open-problems}).

\subsection*{Notation} 
Throughout $\base$ denotes an algebraically closed field of characteristic $0$.
In this paper, we adopt the following notation for quotients.  If $X$ is an algebraic space and $G$
is a group scheme, we denote the quotient stack by $X/G$, not $[X/G]$. If $X$ admits a good quotient
by $G$, or equivalently if $X/G$ admits a good moduli space, then we denote this space $X/\!/G$.
For instance, the GIT quotient of a scheme $X$ by a reductive group $G$ is $X^{\rm ss} /\!/ G$,
where $X^{\rm ss} \subseteq X$ is the open semistable locus, not $X/\!/G$. The second author advocates that the field should adopt these conventions immediately, and free quotient stacks from their years of imprisonment between square-brackets.

\subsection*{Acknowledgements}
We thank Giulio Codogni, Maksym Fedorchuk, Andres Fernandez Herrero, Aaron Landesman, Giovanni
Inchiostro, and Michele Pernice for helpful comments during the preparation of this article.  We
also thank all of our collaborators, especially Jack Hall, Andres Fernandez Herrero, Jochen
Heinloth, Andr\'es Ib\'a\~nez N\'u\~nez, and David Rydh. 

\section{Local Structure of Algebraic Stacks} \label{sec:local-structure}

Quotient stacks form a distinguished class of algebraic stacks.  The geometry of a quotient stack
$X/G$ is particularly well understood---it is simply the $G$-equivariant geometry of $X$, a subject
that has been very well studied over the last century.  As such, quotient stacks provide valuable
geometric intuition for general algebraic stacks.  While many familiar moduli stacks, e.g.,
$\bar{\cM}_g$, $\cBun^{\ss}_{r,d}(C)$, and  $\Fanos^{K-\ss}_{n,V}$ are global quotient stacks, many
other moduli stacks, e.g., the moduli $\cM^{\sigma-\ss}$ of Bridgeland semistable objects with
respect to a stability condition $\sigma$, are not known to be global quotient stacks.  

\subsection{Global quotient stacks}
Certain classes of
stacks are known to be global quotient stacks, e.g., smooth, separated 
Deligne--Mumford (DM) stacks with either generically trivial stabilizer 
\cite[Thm.~2.18]{ekhv}
or quasi-projective coarse moduli space 
\cite[Thm.~4.4]{kresch-geometry-of-DM-stacks},
\cite[Thm.~1.9]{olander-olsson}.

On the other hand, it is actually quite difficult to \emph{prove} that a given
moduli stack is not a quotient stack, related to both arithmetic
properties such as the surjectivity of the Brauer map and geometric properties
such as the existence of faithful vector bundles.  In fact, the only examples we
are aware of are variations of the following: 
\begin{itemize}
  \item the moduli stack $\cM_g^{\pre, \le m}$ parameterizing prestable nodal
  curves with at most $m \ge 2$ nodes \cite[Prop.~5.2]{kresch-flattening} (this
  doesn't even have quasi-affine diagonal), 
  \item a classifying stack $BT$ of a non-split torus $T$ of rank $r \ge 2$ over
  the projective nodal cubic $C = \bP^1 / (0 \sim \infty)$ obtained by gluing
  the split torus $\bG_m^r \times \bP^1$ via a non-trivial involution of the
  fibers over $0$ and $\infty$, 
  \item a $\bG_m$-gerbe over a normal separated surface over $\bC$ corresponding
   to a non-torsion element of $\H^2(X, \bG_m)$
   \cite[II.1.11.b]{grothendieck-brauer}, \cite[Ex.~3.12]{ekhv}, and
  \item a $\bZ/2$-gerbe over the non-separated union $Y \bigcup_{Y \minus 0} Y$
    of the affine cone $Y = \Spec \bC[x,y,z]/(xy-z^2)$ obtained by gluing
    trivial $\bZ/2$-gerbes over $Y$ along a non-trivial
    involution over $Y \minus 0$ \cite[Cor.~3.11]{ekhv}. 
\end{itemize}

In general, the following question appears to be open:

\begin{question} \label{ques:global-quotient}
  Is there a separated Deligne--Mumford stack that is not a global quotient stack?
\end{question}

This question is closely related to Totaro's question \cite[Ques.~1]{totaro}
whether an algebraic stack with finite 
diagonal has the resolution property, i.e., whether every coherent sheaf is the quotient of a vector
bundle.  This question is not even known for normal toric varieties or smooth, separated
algebraic spaces.  For more background and additional references, 
see \cite[\S 7.2]{alper-moduli}.

In the topological setting, the situation is much better understood.  
It essentially follows from the definition
that every \emph{orbispace}, i.e., a
topological stack with proper diagonal (e.g., an orbifold or the analytification of
a separated DM stack), is locally a quotient stack, 
and it is a recent theorem that every orbispace is even a global quotient stack
\cite[Thm.~1.1]{pardon}.  In the differential setting, 
Weinstein's conjecture \cite{weinstein_linearization}---now Zung's Theorem
\cite{zung_proper_grpds}---asserts that differential stacks are locally quotient stacks near points
whose stabilizers are proper Lie groups.

\subsection{Theorems}

A recent breakthrough provides a local structure theorem for algebraic stacks
 analogous to Zung's theorem for differential stacks. 

\begin{theorem}[Local Structure Theorem I] \label{thm:local-structure}
  Let $\cX$ be a quasi-separated algebraic stack of finite type over an algebraically closed field
  $\base$, with affine automorphism groups.  For a point $x \in \cX(\bC)$ with linearly reductive
  stabilizer $G_x$, there exists an étale morphism 
  $$(\Spec A/G_x, w) \to (\cX,x)$$ 
  inducing an isomorphism of
  automorphism groups at $w$ \cite[Thm.~1.1]{ahr}.
\end{theorem}

The upshot is that quotient stacks $\Spec A/G$ by a reductive group $G$ can be viewed as the basic
building block of algebraic stacks near points with reductive stabilizer, just as affine schemes are
the building blocks of schemes and algebraic spaces. See  \cite[\S 7.7]{alper-moduli} for further motivation and a self-contained proof.

There are several variants of the local
structure theorem that hold under more general hypotheses and with different conditions on the
cover.  

\begin{theorem}[Local Structure Theorem II]  \label{thm:local-structure2}
	Let $\cX$ be a quasi-separated algebraic stack, locally 
	of finite type over a noetherian scheme, with affine automorphism groups.  If $x \in \cX$ is a
	closed point with linearly reductive stabilizer, then there exists an étale morphism 
	$(\Spec A/\GL_n, w) \to (\cX,x)$
	inducing an isomorphism of automorphism groups at $w$ \cite[Thm 1.1]{ahr2}.
\end{theorem}

\begin{theorem}[Local Structure Theorem III]\label{thm:local-structure3}
  Let $S$ be an excellent algebraic space and let $\cX$ be an algebraic
  stack, quasi-separated and locally of finite presentation over $S$ with
  affine automorphism groups.  Consider a diagram
  \[\xymatrix{
    \cW_0 \ar@{^(-->}[r] \ar[d]^{f_0}  & \cW \ar@{-->}[d]^f \\
    \cX_0 \ar@{^(->}[r]  & \cX   
  }\]
  where $\cX_0\into \cX$ is a closed immersion, $f_0\colon \cW_0\to \cX_0$ is a morphism of
  algebraic stacks, and $\cW_0$ is isomorphic to a quotient stack $\Spec A_0 / \GL_n$ that is
  cohomologically affine.\footnote{This is automatic in characteristic $0$.}
  If $f_0$ is smooth (resp., étale, syntomic\footnote{For the syntomic result, one needs  to further
  assume that $\cX_0$ has the resolution property along with mild hypotheses on the characteristics,
  e.g., every closed point of $\cW_0$ has positive characteristic.}), then there exists a smooth
  (resp.\ étale, syntomic) morphism $f\colon \cW\to \cX$ such that $\cW \cong \Spec A/\GL_m$  and
  $f|_{\cX_0}\simeq f_0$ \cite[Thm.~1.3]{ahlhr}.
\end{theorem}

\subsection{Applications}
These theorems have had many applications in algebraic geometry, most notably in equivariant
geometry.  For instance,  it has been applied to generalize Sumihiro's theorem for torus actions on
normal schemes to arbitrary schemes and algebraic spaces, where instead of equivariant Zariski open
affine neighborhoods, one obtains equivariant étale affine neighborhoods.  It has also been applied
to prove a general existence result for pushouts for algebraic stacks. The original motivation 
and one of the more striking applications is the Existence Theorem 
(\Cref{thm:existence}) for good moduli spaces.

\section{Beyond Geometric Invariant Theory} \label{sec:beyond-git}

We will introduce and motivate the beyond GIT framework using 
the running example of the moduli stack $\cBun := \cBun(C)$ of vector bundles over a smooth
projective curve $C$ defined over an algebraically closed field $\base$ of characteristic $0$.

\subsection{Semistable vector bundles over a curve}
The stack $\cBun$ assigns any $\base$-scheme $T$ to the groupoid of vector bundles on $C_T := C
\times T$. 
Even after fixing numerical invariants, the moduli stack $\cBun_{r,d} \subseteq \cBun$ of  
vector bundles on $C$ of rank $r$ and degree $d$ is too big to admit a  moduli space: it is both 
unbounded (i.e., not quasi-compact) and highly non-separated.   As suggested by Mumford and
Seshadri, it is natural to restrict to the open substack $\cBun_{r,d}^{\ss} \subseteq \cBun_{r,d}$
of semistable vector bundles.
A vector bundle $F$ on $C$ is  \emph{semistable} if every nonzero subsheaf $E
\subseteq F$ satisfies
\[
  \mu(E):=\frac{\deg E}{\rk E} \le \frac{\deg F}{\rk F} =: \mu(F),  
\] 
where $\mu(F)$ is called the \emph{slope} of $F$.  A remarkable feature of semistability  is that
every vector bundle $F$ has a unique filtration
$$F_p \subsetneq \cdots \subsetneq F_0 = F$$ 
whose associated graded bundles $F_i/F_{i+1}$ are semistable with monotone increasing slope.  
This is called the \emph{Harder--Narasimhan filtration} or simply the \emph{HN filtration}.
Moreover, the substack of bundles whose HN filtration has prescribed ranks and degrees is locally
closed in $\cBun_{r,d}$, and retracts onto the products of stacks of semistable objects
corresponding to the associated graded pieces of the HN filtration.  This gives a stratification of
$\cBun_{r,d}$ called the \emph{HN stratification}. For further details on semistability, see
\cite{le-potier}, \cite{huybrechts-lehn}, or \cite[\S 9]{alper-moduli}. 

\subsection{Geometric invariant theory}
By synthesizing Grothendieck's formalism of scheme theory with 19th century invariant theory,
Mumford developed a theory of quotients in algebraic geometry known as \emph{Geometric Invariant
Theory} (\emph{GIT}), and applied this theory to construct moduli spaces \cite{git}.  The approach
of GIT proceeds along the following three steps:
\begin{enumerate}[(1)]
    \item
      Express the moduli problem $\cM$ as a quotient $U/G$ of the
      action of a quasi-projective variety $U$ by a reductive group
      $G$ and choose a $G$-equivariant compactification $\bar{U}
      \subseteq \CP^n$, where $G$ acts linearly on $\CP^n$.
    \item
      Show that a point $u \in \bar{U}$ is in $U$ if and only if
      it is GIT semistable, i.e., there exists a non-constant
      $G$-invariant homogeneous $f \in \Gamma(\bar{U},
      \oh(d))^G$ with $f(u) \neq 0$. 
    \item
      Realize the moduli space of $\cM$ as the projective
      variety $U \gitq G := \bigoplus_{d \ge 0}
      \Gamma(\bar{U}, \oh(d))^G$.
\end{enumerate}

To apply this approach to $\cM = \cBun^{\ss}_{r,d}$, one utilizes Grothendieck's projective Quot scheme
\[
\bar{U}_n:=\Quot^P(\oh_C^{\oplus P(n)}(-n) / C)
\]
parameterizing coherent quotients $[\pi \co \oh_C^{\oplus
P(n)}(-n) \onto F]$, where $P(n) = r n \deg(\oh_C(1)) + (1-g) r + d$ is the Hilbert polynomial. For $n \gg 0$, the GIT semistable locus $U_n \subseteq \bar{U}_n$ corresponds to quotients such that $F$ is a semistable vector bundle and the induced map $\H^0(\pi(n)) \co \base^{\oplus P(n)} \to \H^0(C, F(n))$ is an isomorphism, and every semistable $F$ appears in $U_n$. This description of the GIT semistable locus, first established by Seshadri \cite{seshadri-space-of-unitary-vector-bundles}, is the hardest part of the GIT approach, and has since been
revisited and generalized many times. (See \cite[\S 9.5]{alper-moduli} for a discussion).
As a consequence, we obtain a projective good moduli space
$M^{\ps}_{r,d} := U_n \gitq \GL_{P(n)}$ for $\cM \cong U_n/\GL_{P(n)}$ whose $\base$-points are in bijective correspondence with
polystable vector bundles.

Another consequence of the general GIT machinery is that the unstable locus $\bar{U} \setminus U$ admits a stratification, referred to as the 
\emph{HKKN stratification} after Hesselink, Kempf, Kirwan, and Ness \cite{kempf1978instability, MR553706}; see \cite[\S 8.7]{alper-moduli}
for a recent exposition and further references. While the strata of $\bar{U} \setminus U$ typically do not admit clean modular interpretations, in the case of $\cBun_{r,d}$ (and sheaves on higher dimensional projective schemes) the GIT stratifications of $\bar{U}_n \setminus U_n$ essentially stabilize as $n \to \infty$ to the Harder--Narasimhan stratification of
$\cBun_{r,d}$ \cite{hoskins-hkkn-vs-hn}.

\subsection{Reinterpreting semistability: $\Theta$-semistability}
We will recast the notion of semistability of vector bundles on $C$ in terms of $\bZ$-weighted
filtrations, which in turn will allow us to generalize the notion of semistability to arbitrary
stacks.  By a \emph{$\bZ$-weighted filtration} $E_{\bullet}$, we mean a filtered vector bundle 
$E_p \subsetneq E_{p-1} \subsetneq \cdots \subsetneq E_0 = E$ and a choice of integers 
$w_0 < \cdots < w_p$. The following lemma turns out to be key to interpreting the semistability
condition \emph{geometrically}.

\begin{lemma} \label{lem:semistability-weighted-filtration}
  A vector bundle $E$ is semistable if and only if 
  $$\sum_{i = 0}^p w_i \big(\deg(E) \rk(E_i/E_{i+1}) -  \rk(E) \deg(E_i/E_{i+1}) \big) \ge 0$$
  for every $\bZ$-weighted filtration $E_{\bullet}$ of $E$. 
\end{lemma}

The expression in the lemma does not appear out of thin air: it occurs as the weight of a line
bundle on $\cBun$ restricted to the associated graded bundle of the $\bZ$-weighted filtration.
Namely, from a $\bZ$-weighted filtration, one obtains a sheaf 
$\sum t^{-w_i} \mathcal{O}_C[t] \cdot E_i \subseteq \oh_{C}[t^{\pm 1}] \otimes E$ 
of graded $\oh_{C}[t]$-modules, where $t$ has weight $-1$.  Fixing $\Theta$ as
the quotient 
$$\Theta := \bA^1/\bG_m,$$ 
$\sum t^{-w_i} E_i$ corresponds to a vector bundle on the stack $C \times \Theta$, or
equivalently to a morphism $f:\Theta \to \cBun$; the map $f:\Theta \to \cBun$
takes $1$ to $E$ and $0$ to the associated graded bundle $\bigoplus_{i=0}^p E_i/E_{i+1}$.
For any line bundle $\cL$ on $\cBun$, we 
define $\wt(f^\ast \cL)$ as the \emph{opposite} of the unique integer $w$ such that 
$f^\ast \cL \cong \oh_{\Theta}\langle w \rangle$.  
We will use the line bundle 
\[
  \cL_{\rm{adj}} := \det(R p_{2,\ast}(E_{\univ} \otimes E_{\univ}^\vee))^\vee
\] 
where $E_{\univ}$ is the universal bundle on  $C \times \cBun$.  Letting $G_i = E_i/E_{i+1}$ be
the graded factors, we can easily compute the weight: 
\begin{align*}
    \wt(f^\ast \cL_{\rm adj}) 
      &= - \wt(\det(H^\ast(C,\bigoplus_{i,j=0}^p G_i \otimes G_j^\vee))) \\
      &= -\sum_{i,j=0}^p (w_i - w_j) \cdot \chi(C, G_i \otimes G_j^\vee) \\
      &= -\sum_{i,j=0}^p (w_i - w_j)  \big(\deg(G_i)\rk(G_j) - \deg(G_j)\rk(G_i)\big) 
      & \text{(Riemann--Roch)}\\
      &= -2 \sum_{i=0}^p w_i  (\rk(E)\deg(G_i) - \deg(E) \rk(G_i)) \\
      &=-2 \sum_{i=1}^p (w_i - w_{i-1})(\rk(E) \deg(E_i) - \deg(E) \rk(E_i)) 
      &  \text{(telescoping sum)}\\
      & = 2 \sum_{i = 1}^p (w_i - w_{i-1}) \rk(E) \rk(E_i) \big(\mu(E) - \mu(E_i) \big)
\end{align*}
The expression in \Cref{lem:semistability-weighted-filtration} is one half of the fourth expression.
The final expression makes it transparent that $E$ is semistable if and only if $\wt(f^\ast \cL_{\rm
adj}) \geq 0$ for all maps $f \co \Theta \to \cBun_{r,d}$ with $f(1) = [E]$.  Except for the factor
of 2, the same calculation holds for the line bundle $\cL_V := \det(R p_{2,\ast}(E_{\univ} \otimes
V))^\vee$, where $V$ is a vector bundle on $C$ of rank $r$ and degree $r(g-1) - d$.

Semistability can be formulated in this way for any stack equipped with a line bundle \cite{halpernleistner2022structureinstabilitymodulitheory, heinloth-stability}:

\begin{definition}\label{D:theta_stability}
    A point $x \in |\cX|$ is \emph{semistable with respect to $\cL$} if $\wt(f^\ast \cL) \geq 0$ for
    all morphisms $ f \co \Theta_K \to \cX$ (where $K$ is a field extension of $\base$) with $f(1)
    \cong x$.  We denote by $\cX^{\cL-\ss}$ the substack of points semistable with respect to
    $\cL$.
\end{definition}

\begin{remark}
    The sign conventions for the weight of a line bundle on $\Theta$ can be confusing.  Our 
    guiding convention is that a map $f: \Theta_K \to \cX$ is destabilizing if 
    $\Gamma(\Theta_K,f^\ast(\cL)) = 0$.
\end{remark}

The above calculations show that $\cBun^{\ss} = \cBun^{\cL_{\rm adj}-\ss}$. Also note that when
$\cL= \oh_{\cX}$ is trivial, then $\cX^{\ss} = \cX$. The above definition
will be subsumed by an even more general notion of a numerical invariant on $\cX$ below, but this
suffices for the construction of moduli spaces.

\begin{example}[GIT]
  Consider a quotient stack $\cX = X/G$ with $G$ reductive and a $G$-linearized line bundle $\cL$ on $X$.  The definition of \emph{GIT semistability}  \cite[Def.~1.7]{git} for a point $x \in X(\base)$ requires that there be a section $s \in \H^0(X, \cL^{\otimes n})^G$ for some $n > 0$ such that $s(x) \neq 0$ and $X_s$ is affine.  When $X$ is projective and $\cL$ is ample, the affineness of $X_s$ is
  automatic.  Moreover, in this case, the Hilbert--Mumford Criterion \cite[Thm.~2.1]{git} asserts
  that $\cX^{\cL-\ss} = X^{\ss}/G$.   
  
  One can also define an analogous notion of GIT semistability for a polarized algebraic stack
  $(\cX, \cL)$   \cite[Defn.~11.1]{alper-good} by requiring the existence of a section
  $s \in \H^0(\cX, \cL^{\otimes n})$ for some $n > 0$ such that $s(x) \neq 0$ with $\cX_s$
  cohomologically affine.  With this definition, the semistable locus is necessarily open and admits
  a quasi-projective good moduli space, but it appears difficult in practice to determine the semistable locus, even its non-emptiness. 
    
  If $X$ is not projective or $\cL$ is not ample, then $\cX^{\cL-\ss}$ need not be open
  in $\cX$, e.g., $\cX = (\bA^2 \setminus 0) / \bG_m$ with weights $0$, $1$ and $\cL =
  \oh_{\cX}\langle d \rangle$ with $d < 0$ in which case $\cX^{\cL-\ss}$ is the $y$-axis, or $\cX =
  {\rm Bl}_0 \bP^2 / \bG_m$ where $\cL$ is the pullback of $\oh(1)$ from $\bP^2/\bG_m$ in which case
  $\cX^{\cL-\ss}$ is the exceptional divisor. 
\end{example}

\begin{example}[K-semistability]\label{E:K-stability}
  Let $\Fanos_{n,V}$ be the algebraic stack parameterizing families $X \to S$ of $\bQ$-Fano klt
  varieties of dimension $n$ and volume $V$  satisfying Kollár's condition. 
  If $\cL_{\rm CM}$ is the CM line bundle, then the $\cL_{\rm CM}$-semistable locus
  $\Fanos_{n,V}^{\cL_{\rm CM}-\ss}$ is precisely the locus  $\Fanos_{n,V}^{K-\ss}$ of $K$-semistable
  $\bQ$-Fano varieties. As the culmination of several major breakthroughs in K-stability over the
  past decade, we now know that $\Fanos_{n,V}$ admits a $\Theta$-stratification (see \S \ref{subsec:theta-stratifications}) and that the 
  semistable locus admits a projective good moduli space $\Fanos_{n,V}^{K-\ss}
  \to M^{\rm ps}_{n,V}$ parameterizing $K$-polystable $\bQ$-Fano varieties. See \cite{xu-book}. 
\end{example}

\subsection{Good moduli spaces}

If $G$ is a (not necessarily connected) reductive group acting on an affine variety $X = \Spec A$,
then we define the GIT quotient as $X /\!/ G := \Spec A^G$. If every stabilizer is finite, then
$X/\!/G$ is a geometric quotient, i.e., the morphism of stacks $X/G \to X/\!/G$ is universally
bijective.  In general, $X/\!/G$ may identify distinct orbits, e.g., $\bA^1/\bG_m \to \Spec \base$,
but nevertheless $X/\!/G$ is the space that best approximates the stack $X/G$. This motivates the
following a-historical
definition:

\begin{definition}\label{D:gms}
    A morphism $q \co \cX \to X$, from an algebraic stack $\cX$ of finite type over an algebraically
    closed field $\base$ of characteristic $0$ with affine diagonal to an algebraic space $X$ over
    $\base$, is a \emph{good moduli space} if there is an \'{e}tale cover of $X$ of the form
\[
  \xymatrix{
    \Spec A/G \ar[d] \ar[r]^(.65){\text{étale}}
      & \cX \ar[d]^q \\
    \Spec A^G \ar[r]^(.6){\text{\'etale}}
      & X, \cartesian
  }
\]
where $G$ is a reductive group and the diagram is cartesian.\footnote{Since $\base$ has 
characteristic $0$, every reductive group is  linearly reductive.  In positive or mixed characteristic, this definition provides an example of an
\emph{adequate moduli space}, but precisely because Local Structure Theorem (\Cref{thm:local-structure2}) is unknown in this
setting (see \Cref{conj:local-structure-charp}), it is unclear if this is equivalent to the original
definition of an adequate moduli space given in \cite{alper-adequate} (even under the finite
typeness and affine diagonal hypotheses).}  

\end{definition}

A good moduli space morphism enjoys the following useful properties (see \cite[\S 7.5]{alper-moduli} for a recent exposition): 
\begin{enumerate}[(1)]
    \item $q_\ast \co \QCoh(\cX) \to \QCoh(X)$ is exact, and $q_\ast \oh_\cX \cong
    \oh_X$.\footnote{This was actually the original definition \cite[Defn~4.1]{alper-good}, and this condition alone implies all
    of the other properties. The Local Structure Theorem (\Cref{thm:local-structure2}) implies that (1) is equivalent to
    \Cref{D:gms}.}
    \item Any morphism to an algebraic space $\cX \to Y$ factors uniquely through $q$.
    \item $q$ is universally closed, and $q_\ast$ preserves coherent sheaves.
    \item Every fiber of $q$ has a unique closed point $x_0$. The automorphism group $\Aut(x_0)$ 
    is reductive, and for any other $\base$-point $y \in q^{-1}(q(x_0))$  there is a morphism 
    $f\co \Theta \to \cX$ with $f(1) \cong y$ and $f(0) \cong x_0$, which by Kempf's optimal
    destabilization theorem \cite{kempf1978instability}
    is canonical up to the choice of a norm on one-parameter subgroups
    of $G$.
\end{enumerate}

Since $\cX$ has affine diagonal, the fiber product $(\Spec A /G) \times_\cX (\Spec A/G)$ has the form
$\Spec B/G$ for some $G$-equivariant algebra $B$, and $X$ can be constructed as the quotient of
$\Spec A^G$ by the \'{e}tale equivalence relation 
$\Spec B^G \hookrightarrow \Spec A^G  \times \Spec A^G$. 
Therefore, the existence of a good moduli space $\cX \to X$ is equivalent to the
existence of a representable \'{e}tale cover $\Spec A /G \to \cX$ such that both induced morphisms
$\Spec B^G \to \Spec A^G $ are \'{e}tale. Clearly then, what one needs to construct good moduli
spaces are conditions that allow one to apply the Local Structure Theorem
(\Cref{thm:local-structure}), and to guarantee that the resulting cover satisfies this additional
\'{e}tale hypothesis.

Given a discrete valuation ring (DVR) $R$ with uniformizer $\pi$, we define $\Theta_R := \Spec(R[t])
/ \bG_m$ and $\ST_R := \Spec(R[s,t]/(st-\pi)) / \bG_m$, with the $\bG_m$-action determined
by giving $t$ weight $-1$ and $s$ weight $1$ in both instances. Both stacks have a unique closed
point, with maximal ideal generated by $(\pi,t)$ and $(s,t)$ respectively, which we call $0$. 

\begin{definition}
  We say that $\cX$ is \emph{$\Theta$-complete} (formerly called \emph{$\Theta$-reductive}) if for any DVR $R$ over $\base$, any
  morphism $\Theta_R \setminus 0 \to \cX$ extends uniquely to $\Theta_R$.
  Likewise, $\cX$ is \emph{$\sS$-complete} if for any DVR $R$ over $\base$, any morphism
  $\ST_R \setminus 0 \to \cX$ extends uniquely to $\ST_R$.
\end{definition}

\begin{theorem}[Existence Theorem]\label{thm:existence}
  Let $\cX$ be an algebraic stack of finite type over an algebraically closed field $\base$ 
  of characteristic $0$ with affine diagonal. Then there
  exists a good moduli space $\cX \to X$ where $X$ is a separated algebraic space if and only if
  $\cX$ is $\Theta$-complete and
  $\sS$-complete \cite[Thm.~A]{ahlh-existence} (see also \cite[\S 7.10]{alper-moduli}).
\end{theorem}

\subsection{Beyond GIT: verifying $\Theta$- and $\sS$-completeness} \label{sec:beyond-git-completeness}

We sketch two ways to verify that the stack $\cBun^{\rm{ss}}_{r,d}$ of semistable bundles on $C$ is
$\Theta$- and $\sS$-complete, which by the Existence Theorem (\Cref{thm:existence}) guarantees the 
existence of a separated good moduli space.  Both approaches have been vastly generalized, and are 
still the subject of ongoing
research. They each require that $\cBun^{\ss}_{r,d}$ is open and bounded as input,
so we take these as given. We will return to the openness question when we discuss
$\Theta$-stratifications.

To make matters concrete for $\cBun$: We are given a $\bG_m$-equivariant vector bundle $E$ on
$C_{Y\setminus 0} := C \times (Y \setminus 0)$, and the question is under what conditions we can
extend $E$ to $C_Y$, where $Y$ is either $\Spec(R[t])$ or $\Spec(R[s,t]/(st-\pi))$. If such an
extension exists, it must be isomorphic to $j_\ast E$, where $j \co C_{Y\setminus 0} \into C_Y$ is
the open immersion.

\medskip
\noindent{\textbf{First approach:}}
If $\cX$ is $\Theta$- and $\sS$-complete, then the semistable locus
$\cX^{\cL-\ss}$ with respect to any line bundle $\cL$ is also $\Theta$- and $\sS$-complete 
\cite[Prop.~6.14]{ahlh-existence}. 
However, the stack $\cBun$  is neither $\Theta$- nor $\sS$-complete, so an initial step 
is to find an enlargement
$\cBun \subseteq \cX$ such that $\cX$ is $\Theta$- and $\sS$-complete. Reflecting back on the above
setting, observe that the pushforward $j_\ast E$ is automatically a $Y$-flat equivariant family of
coherent sheaves, but that the zero fiber $j_\ast(E)|_{C_0}$ may have torsion. In other words, the
larger stack $\cCoh(C)$ of coherent sheaves on $C$ is $\Theta$- and $\sS$-complete. 
It is in fact a general phenomenon that moduli stacks of objects in an abelian category are 
$\Theta$- and $\sS$-complete \cite[\S 7]{ahlh-existence}. The line bundle $\cL_{\rm{adj}}$ naturally
extends to $\cCoh(C)$, and it is
fairly easy to verify  that a bundle is semistable as a point in $\cBun(C)$ if and only if it is
semistable in $\cCoh(C)$.

\medskip
\noindent{\textbf{Second approach:}}
The following ideas were developed in \cite[\S 5]{halpernleistner2022structureinstabilitymodulitheory}, and receive an expository treatment in \cite{modulibook}.

\begin{definition}\label{D:monotone_line_bundle}
    A line bundle $\cL$ on an algebraic stack $\cX$ 
    is \emph{$\Theta$-monotone} (resp., \emph{$\sS$-monotone}) if for any
    morphism $f \co (Y \setminus 0)/\bG_m \to \cX$, there is a proper birational morphism $\Sigma /
    \bG_m \to Y/\bG_m$ and an extension $\tilde{f} \co \Sigma/\bG_m \to \cX$ such that
    $\tilde{f}^\ast(\cL)$ is relatively ample over $Y/\bG_m$.
\end{definition}

Note that this condition holds automatically if $\cX$ is $\Theta$- and $\sS$-complete, and
hence this approach can be viewed as a generalization of the first approach.

\begin{proposition}
    If $\cL$ is a $\Theta$- and $\sS$-monotone line bundle on an algebraic stack $\cX$, then
    $\cX^{\cL-\ss}$ is $\Theta$- and $\sS$-complete.
\end{proposition}

Returning to $\cX = \cBun$ and a vector bundle $E$ on $C \times (Y \setminus 0)$, the pushforward
$j_\ast E $ is locally free away from finitely many points in the fiber $C_0$. After
removing a $\bG_m$-invariant horizontal divisor $D \subseteq C_Y$ containing these points, one can
construct an isomorphism of equivariant sheaves $j_\ast(E)|_{C_Y \setminus D} \cong F|_{C_Y
\setminus D}$, where $F = \oplus_i \oh_Y\langle w_i \rangle$ is a sum of trivial bundles twisted by
characters of some weights $-w_i$. This isomorphism allows one to sandwich $F|_{C_{Y\setminus
0}}(-nD) \subseteq E \subseteq F|_{C_{Y\setminus 0}}(nD)$ for some $n \gg 0$.

We now observe that the data of a flat family of sheaves $E'$ on $C$ with $F(-nD) \subseteq E' \subset
F(nD)$, which is automatically locally free, is the same as a flat family of quotients of the family
of $0$-dimensional sheaves $F(nD)/F(-nD) \to Q$, where $Q$ is mapped to its preimage under the
quotient morphism $F(nD) \to F(nD)/F(-nD)$. This defines a morphism $\Quot(C_Y, F(nD)/F(-nD)) \to
\cBun$, and crucially, the pullback of $\cL_{\rm{adj}}$ is ample relative to $Y$. 
The bundle $E$ determines a
section over $Y \setminus 0$ of the equivariant proper morphism $\Quot(C_Y, F(nD)/F(-nD)) \to Y$. We
let $\Sigma \hookrightarrow \Quot(C_Y, F(nD)/F(-nD))$ denote the closure of this section, which is a
proper $Y$-scheme, and the composition $\Sigma /\bG_m \to \Quot(F(nD)/F(-nD))/\bG_m \to \cBun$
satisfies \Cref{D:monotone_line_bundle}.

The appearance of the Quot scheme in this second approach is very different than the Quot scheme in
GIT. Instead, $\Quot(F(nD)/F(-nD))$ is an approximation of a Beilinson--Drinfeld Grassmannian, which
is an ind-projective ind-scheme. It can be useful to think of this approach as an infinite
dimensional generalization of GIT. This infinite dimensional GIT was first developed in \cite{hl-fh-gauged-maps, hl-fh-j}, where it is shown to generalize naturally to other ``non-linear'' moduli
problems, such as the moduli of generalized Bradlow pairs, or the stack of maps $C \to X/G$ from a smooth curve $C$.

\subsection{\texorpdfstring{$\Theta$}{Theta}-stratifications} \label{subsec:theta-stratifications}
For any stack $\cX$, we define the stack of filtered points of $\cX$ to be $\Filt(\cX) :=
\Map(\Theta,\cX)$. This is equipped with an evaluation map ${\rm ev}_1 \co \Filt(\cX) \to \cX$
taking $f \co \Theta \to \cX$ to $f(1)$.  When $\cX = \cBun_{r,d}$, a $\base$-point $f \in
\Filt(\cX)$ is a $\bZ$-weighted filtration 
$E_p \subsetneq \cdots \subsetneq E_0 = E$ with  
weights $w_0 < \cdots < w_p$, and $\ev_1(f) = [E]$ for this $f$.  It is often convenient to work instead
with the stacks of $\bQ$-weighted
filtered points, $\Filt_\bQ(\cX)$, obtained by stabilizing under certain morphisms $\Filt(\cX) \to
\Filt(\cX)$ that scale weights \cite[Prop.~1.3.11]{halpernleistner2022structureinstabilitymodulitheory}.

We can  regard a Harder-Narasimhan filtration  $E_p \subsetneq \cdots \subsetneq E_0 = E$ of a
vector bundle as a point in $\Filt_\bQ(\cBun_{r,d})$ by equipping it with the weights $w_i :=
\mu(E_i/E_{i+1}) - \mu(E) \in \bQ$. The set of all Harder-Narasimhan filtrations defines a subset
$\cS \subseteq\Filt_{\bQ}(\cX)$, whose key properties are generalized in the following:

\begin{definition}[Variant of {\cite[Def.~2.1.2]{halpernleistner2022structureinstabilitymodulitheory}}]
    A \emph{$\Theta$-stratification} of a stack $\cX$ is an open substack $\cS \subseteq
    \Filt_\bQ(\cX)$ along with a total preordering on the set of connected components $\pi_0(\cS)$
    such that $\ev_1 \co \cS \to \cX$ is universally bijective, and $\forall \gamma \in \pi_0(\cS)$,
    \begin{enumerate}
        \item $\cX_{\leq \gamma} := \cX \setminus \bigcup_{\gamma' > \gamma} \ev_1(\cS_{\gamma'})
        \subseteq\cX$ is open, hence an open substack,
        \item $\cS_\gamma \subseteq\Filt_\bQ(\cX_{\leq \gamma}) \subseteq\Filt_\bQ(\cX)$, and
        \item $\ev_1 \co \cS_\gamma \to \cX_{\leq \gamma}$ is a closed immersion.
    \end{enumerate}
\end{definition}

For two components $\gamma, \gamma'$ of $\Filt_\bQ(\cBun_{r,d})$, we say $\gamma' > \gamma$ if
$\nu(\gamma')>\nu(\gamma)$, where $\nu \co \Filt_\bQ(\cX) \to \bR$ is the locally constant function
that assigns a weighted filtration $f = (E_\bullet,w_\bullet)$ with associated graded pieces $G_i :=
E_i / E_{i+1}$ the value
\begin{equation}\label{E:example_numerical_invariant}
    \nu(f) = \sum_i \rk(G_i) \cdot 
      \left( -w_i \left(\mu(G_i) - \mu(E) \right) +  \frac{1}{2} w_i^2 \right).
\end{equation}
This preordering restricted to $\cS$ defines a $\Theta$-stratification. Remarkably, the function
$\nu$ also identifies the open substack $\cS \subseteq \Filt_\bQ(\cX)$ of Harder-Narasimhan
filtrations: If $M^\nu \co |\cBun_{r,d}| \to \bR$ is given by
\begin{equation} \label{eqn:Mnu}
M^\nu(x) = \inf \{\nu(f) | f \in \Filt_\bQ(\cBun_{r,d}) \text{ s.t. }f(1) \cong x\},
\end{equation}
then $f \in |\Filt_\bQ(\cBun_{r,d})|$ lies in $\cS$ if and only if $\nu(f) = M^{\nu}(f(1))$.

An important application of $\Theta$-stratifications is to study wall-crossing behavior: as one varies a line bundle or numerical invariant, the semistable locus $\cX^{\ss}$ typically undergoes discrete changes at ``walls'' in the space parameterizing these stability conditions. If $\cX^{\ss}$ is the open piece of a $\Theta$-stratification of $\cX$, then one can often say which strata remain unstable and which strata pick up semistable points as one crosses a wall. A particularly clean version of this is developed in \cite[\S 5.6]{halpernleistner2022structureinstabilitymodulitheory} for a stack $\cX$ that admits a quadratic norm on graded points and a good moduli space $\cX \to X$. Any class in $\rm{NS}(\cX)_\bQ/\rm{NS}(X)_\bQ$ defines a $\Theta$-stratification of $\cX$ whose semistable locus admits a good moduli space that is projective over $X$, and the unstable strata precisely encode where the semistable locus can jump --- this ``variation of good moduli space'' is an intrinsic version of the classical theory of variation of GIT quotient \cite{dolgachev98variation}.

One kind of result along these lines is to use the semiorthogonal decompositions of $\rm{D}^-_{\rm{coh}}(\cX)$ associated to a $\Theta$-stratification \cite{halpernleistner2021derivedthetastratificationsdequivalenceconjecture} to study how the derived category of the semistable locus $\cX^{\rm ss}$ varies under wall-crossing. The structure theory for $\rm{D}^-_{\rm{coh}}(\cX)$ also leads to $K$-theoretic wall-crossing formulas \cite[\S 4.1]{halpernleistner2025categoricalperspectivenonabelianlocalization}. In a different direction, this variation of good moduli space picture has been used to formulate an ``intrinsic'' version of DT theory, which applies to more general moduli problems than traditional DT theory, and to prove wall-crossing formulas in that setting \cite{bu2025intrinsicdonaldsonthomastheoryi, bu2025intrinsicdonaldsonthomastheoryii}. 

Another use for $\Theta$-stratifications in moduli theory is the following:
\begin{theorem}[Semistable reduction {\cite[Thm.~6.5]{ahlh-existence}}]
    Suppose that $\cS \subseteq\Filt_\bQ(\cX)$ is a $\Theta$-stratification such that $\pi_0(\cS)$ is
    well-ordered. Then for any DVR $R$ and morphism $\xi \co \Spec R \to \cX$,
    there is a finite extension $R \subseteq R'$ and a morphism $\xi' \co \Spec R' \to \cS$ such that
    the generic point of $\xi'$ is a lift of the generic point of $\xi$ under 
    $\ev_1 \co \cS \to \cX$.
\end{theorem}
In particular, if $\cX$ satisfies the existence part of the valuative criterion for properness, then
so does $\cS$. If in addition $\cX^{\rm ss} = \cS \cap \cX$ admits a finite type and separated good
moduli space $M$, then $M$ is proper. 

There is a list of necessary and sufficient conditions for an open substack $\cS \subset
\Filt_\bQ(\cX)$ to define a $\Theta$-stratification \cite[Thm.~2.2.2]{halpernleistner2022structureinstabilitymodulitheory}. On the other hand, the theory of
numerical invariants, which we describe next, allows one to define a $\Theta$-stratification without
first knowing what the HN filtrations should be. Our definition here is simpler than the one in \cite[Def.~4.1.1]{halpernleistner2022structureinstabilitymodulitheory}, because we will not be discussing the concept in enough detail to necessitate the general framework.

\begin{definition}
    A \emph{numerical invariant} on a stack $\cX$ is a locally constant function $\nu \co
    \Filt_\bQ(\cX) \to \bR$. 
\end{definition}

Given a numerical invariant $\nu$ on $\cX$, we can define $M^\nu \co |\cX| \to \bR$ as in
\eqref{eqn:Mnu}. We say
that $\nu$ \emph{defines a $\Theta$-stratification} if the set
\[
\cS := \{f \in |\Filt_\bQ(\cX)| \mid \nu(f) = M^\nu(f(1)) \}
\]
is open in $\Filt_{\bQ}(\cX)$ and is a $\Theta$-stratification after preordering $\pi_0(\cS)$ by
the values of $\nu$. One can canonically regard $\cX \subseteq \Filt_\bQ(\cX)$ as the open and
closed substack parameterizing filtrations $f$ such that $\bG_m \to \Aut_\cX(f(0))$ is trivial, and
we refer to $\cX^{\nu-\ss}:=\cS \cap \cX \subseteq\cX$ as the \emph{semistable locus}.  Observe that
$x \in \cX^{\nu-\ss}$ if and only if $\nu(f) \geq 0$ for all filtrations $f$ with $f(1) \cong x$.

Let us consider the optimization problem $M^\nu(x)$ on $\cBun_{r,d}$. For each unweighted filtration
$E_p \subsetneq \cdots \subsetneq E_0$ of vector bundles, \eqref{E:example_numerical_invariant} can
be regarded as a function on the rational polyhedral cone
\[
\sigma_{E_\bullet}^\circ = \{ w_1 \leq w_2 \leq \cdots \leq w_p \} \subseteq\bR^p_{E_\bullet}.
\]
The points where all $w_i \in \bQ$ and the inequalities are strict correspond to points of
$\Filt_\bQ(\cBun_{r,d})$. The boundary faces of $\sigma_{E_\bullet}$ can be identified with
$\sigma_{\tilde{E}_\bullet}$, where $\tilde{E}_\bullet$ is the filtration obtained from $E_\bullet$
by deleting the step $E_i$ in the filtration if $w_i = w_{i+1}$.

In general, the optimization problem for identifying HN filtrations has this same form:
$\bQ$-weighted filtrations are the rational points of a space, called the \emph{degeneration space of $x
\in \cX(\base)$}, obtained by gluing rational polyhedral cones, and the problem is to find a unique
minimizer of the function $\nu$ on this space. 
There are two key properties for a numerical invariant $\nu$ on an algebraic stack $\cX$:  
\begin{itemize}
  \item \textbf{convexity:} $\nu$ is a strictly convex function on each of these cones
  \item \textbf{condition (R):} $\nu$ achieves a minimum at a rational point on each of these cones
\end{itemize}
Both of these conditions are satisfied for the numerical invariant
\eqref{E:example_numerical_invariant} on $\cBun_{r,d}$, and more generally for any numerical invariant coming from a class in $\rm{NS}(\cX)_\bQ$ and a rational quadratic norm on graded points on $\cX$. See \cite[\S 4.1, \S 6]{halpernleistner2022structureinstabilitymodulitheory} and \cite{modulibook}.

\begin{theorem}\label{T:existence_theta_stratifications} \cite[Thm.~4.5.1]{halpernleistner2022structureinstabilitymodulitheory}
If $\cX$ is an algebraic stack locally of finite type and with affine automorphism groups and
separated inertia, and $\nu$ is a convex numerical invariant on $\cX$ satisfying condition (R), then
$\nu$ defines a $\Theta$-stratification of $\cX$ if and only if the following two conditions hold:
\begin{enumerate}
    \item \textbf{HN boundedness:} for any bounded open substack $\cU \subseteq\cX$, there is a
    second bounded substack $\cV\subseteq\cX$ such that for any destabilizing filtration $f$ of $x
    \in \cU$, there is a second filtration $f'$ of $x$ with $f(0) \in \cV$ and $\nu(f') \leq \nu(f)$.\smallskip

    \item \textbf{HN specialization:} If $x, y \in |\cX|$ and $x$ specializes to $y$, then
    $M^{\nu}(y) \leq M^\nu(x)$, and if equality holds, then any HN filtration of $x$ specializes in
    $\Filt(\cX)$ to an HN filtration of $y$.
\end{enumerate}
\end{theorem}

There is a notion of $\Theta$- and $\sS$-monotonicity for a numerical invariant $\nu$ that is
analogous to \Cref{D:monotone_line_bundle} \cite[Def.~,5.2.1,5.4.7]{halpernleistner2022structureinstabilitymodulitheory}. It states that with $Y = \Spec(R[t])$ or
$\Spec(R[s,t]/(st-\pi))$ as above, any morphism $(Y \setminus 0) / \bG_m \to \cX$ extends to some
$\Sigma / \bG_m \to \cX$ in such a way that for any $\bG_m$-equivariant quasi-finite morphism
$\bP^1_{\base} \to \Sigma_0$ to the fiber of $\Sigma$ over $0 \in Y$, one has
\[
  \nu (\Theta_{\base} \to \{\infty\}/\bG_m \hookrightarrow \bP^1_{\base} / \bG_m \to \cX) 
    > \nu(\Theta_{\base} \to \{0\}/\bG_m \hookrightarrow \bA^1_{\base} / \bG_m \to \cX).
\]
One can use the infinite dimensional GIT technique, which we previously used to show that
$\cL_{\rm{adj}}$ is monotone on $\cBun_{r,d}$, to show that the numerical invariant
\eqref{E:example_numerical_invariant} is $\Theta$- and $\sS$-monotone on $\cBun_{r,d}$. The
HN-specialization condition for $\cBun_{r,d}$ therefore follows from the following general result:
\begin{proposition} 
    If $\nu$ is a $\Theta$- and $\sS$-monotone numerical invariant on $\cX$, then it satisfies the
    HN-specialization property and the semistable locus of $\cX$ is $\Theta$- and $\sS$-complete \cite[Thm.~5.4.8]{halpernleistner2022structureinstabilitymodulitheory}.
\end{proposition}

The upshot of this whole discussion is that for a convex monotone numerical invariant satisfying
condition (R), all one must do is analyze two kinds of boundedness questions: HN-boundedness implies
the existence of a $\Theta$-stratification, which in particular implies that the semistable locus is
open, and boundedness of the semistable locus then implies the existence of a separated good moduli
space.

\begin{example}[HN boundedness for $\cBun_{r,d}$]\label{E:HN_boundedness} The function
$\nu$ on each cone $\sigma_{E_\bullet}$ extends naturally to the ambient space
$\bR^p_{E_\bullet}$. One can optimize each summand of $\nu$ in \eqref{E:example_numerical_invariant} 
independently, so the unique unconstrained minimum occurs where $w_i = \mu(G_i) - \mu(E)$ for all
$i$. If this point lies outside of $\sigma^\circ_{E_\bullet}$, which happens precisely if $\mu(G_i)$
is \emph{not} increasing in $i$, then the constrained minimizer of $\nu$ lies on the boundary of
$\sigma_{E_\bullet}$ and therefore lies on $\sigma_{\tilde{E}_\bullet}$ for some filtration
$\tilde{E}_\bullet$ obtained by deleting steps of $E_\bullet$. It follows that the value of $\nu$ on
any weighted filtration is dominated by the value on some filtration that is convex in the sense
that $\mu(G_1)<\cdots<\mu(G_p)$. HN boundedness follows from this analysis combined with the
following geometric input: the set of convex filtrations of a bounded family of vector bundles is
bounded. This is worked out in greater generality in \cite{hl-fh-j}.
\end{example}

After this argument for the existence of a maximizer of $\nu$, a small argument is still needed to show that the maximizer agrees with the HN filtration, originally defined as the unique convex filtration whose associated graded pieces are semistable. (See \cite[\S 6]{halpernleistner2022structureinstabilitymodulitheory}.) This alternative characterization of the maximizer of $\nu$ actually holds in much greater generality. Any numerical invariant on a stack $\cX$ induces a numerical invariant on the stack of \emph{graded points} $\Grad(\cX) := \Map(B\bG_m, \cX)$. The \emph{Recognition Theorem} \cite[Thm.~5.4.4]{halpernleistner2022structureinstabilitymodulitheory} says that under fairly general hypotheses, a filtration $f : \Theta_{\base} \to \cX$ is a HN filtration for $f(1)$ if and only if $f|_{\{0\}/\bG_m} : B(\bG_m)_{\base} \to \cX$ is semistable as a point of $\Grad(\cX)$.

Typically each step of the HN filtration of a vector bundle is constructed one-by-one, using the original definition of HN filtrations.
\Cref{E:HN_boundedness} amounts to a conceptually different proof that generalizes to ``non-linear'' moduli problems such as the moduli of
generalized Bradlow pairs on higher dimensional schemes \cite{hl-fh-j}, and the mapping stack
$\Map(C,X/G)$ where $C$ is a smooth curve, $G$ is a reductive group, and $X$ is an affine $G$-scheme
\cite{hl-fh-gauged-maps}. In both cases, the HN filtrations are not convex and cannot be constructed
step-by-step, so the approach above is the only one we know that works.

\begin{remark}
    The original formulation of the theory of $\Theta$-stability defined a numerical invariant as a
    scale-invariant function on filtrations, which leads to HN filtrations being defined only up to
    positive scale. This also requires one to treat the semistable locus separately, as opposed to
    here where it is just a stratum $\cS \subseteq \Filt(\cX)$ that happens to parameterize trivial
    filtrations. The formulation we present here is part of work-in-progress between the second
    author and Andr\'es Ib\'a\~nez-N\'u\~nez. 
\end{remark}

\section{One parameter degenerations in algebraic stacks} \label{sec:combinatorial-structures}

In order to study the HN problem, i.e., minimizing a numerical invariant over all filtrations of a point in $\cX$, \cite[\S 3]{halpernleistner2022structureinstabilitymodulitheory} introduces the \emph{degeneration space} $\Deg(\cX,x)$ associated to a point $x \in \cX(\base)$ for some field $\base$. It is constructed as a union of simplices glued together along rational linear maps, and its rational points parameterize filtrations of $x$ up to positive scaling, i.e., precomposition with a ramified covering $(-)^n : \Theta \to \Theta$. There is another (typically much smaller) space called the \emph{component space} $\Comp(\cX)$ and a continuous map $\Deg(\cX,x) \to \Comp(\cX)$ for any $\xi$. A numerical invariant $\nu$ induces a continuous function on $\Comp(\cX)$, and the HN problem is to minimize the pullback of this function to $\Deg(\cX,x)$. The existence of HN filtrations translates into boundedness properties of $\Comp(\cX)$, and the uniqueness of HN filtrations translates into certain convexity properties $\Deg(\cX,x)$ and the function induced by $\nu$.

This section describes work in progress by the second author to generalize these objects. More precise statements will appear with the completed work. We include a sketch here, because it provides a more conceptual understanding of $\Theta$- and $\sS$-completeness, which has been informing our understanding since 2018. Throughout this section, $R$ will denote a complete DVR with field of fractions $K$ and maximal ideal $(\pi)$.

We will introduce the \emph{affine degeneration space}, which is a metric space $\ADeg(\cX,\xi)$ associated to a point $\xi \in \cX(K)$. Its ``rational points'' parameterize all possible extensions of $\xi|_{K[\pi^{1/m}]}$ to $\cX(R[\pi^{1/m}])$ for some $m>0$. Algebro-geometric properties of $\cX$ translate into metric properties of $\ADeg(\cX,\xi)$, and metric geometry can then be used to deduce interesting algebro-geometric facts.

\subsection{The construction}

We regard the following quotient stack as an algebraic model for the standard $n$-simplex $\Delta^n \subseteq \bR^{n+1}$:
\[
\ST_R^n  :=  \Spec \big( R[t_0^r,\ldots,t_n^r, \forall r \in \bQ_{\geq 0}] / (t_0 \cdots t_n - \pi) \big) / \bG_n 
\]
where $\bG_n$ is the affine group scheme which is Cartier dual to the abelian group $W_n \subseteq \bQ^{n+1}$ consisting of tuples such that $r_0 + \cdots + r_n = 0$. Concretely, $\bG_n = \Spec(\bZ[z^r, \forall r \in W_n])$ with the comultiplication ring homomorphism defined by $z^r \mapsto z^r \otimes z^r$.

A morphism $\ST_R^0 \to \cX$ corresponds to a map $\Spec(R[\pi^{1/m}]) \to \cX$ for some $m$. In addition, any rational linear map $\phi: \Delta^m \to \Delta^n$, defined by a matrix with rational coefficients $\phi_{i,j} \geq 0$ such that $\phi_{i,0} + \cdots + \phi_{i,n} = 1, \forall i=0,\ldots,m$, induces a canonical map $\ST_R^m \to \ST_R^n$ given in coordinates by $t^r_j \mapsto t_0^{r \phi_{0,j}} \cdots t_m^{r \phi_{m,j}}$. So we think of a morphism $\ST_R^n \to \cX$ as giving a family of maps from (totally ramified covers of) $\Spec R$ parameterized by rational points of $\Delta^n$.

The open substack where $t_0\cdots t_n \neq 0$ is isomorphic to $\Spec K'$, where $K' = K[\pi^{1/m}, \forall m>0]$, and we call this generic point $\eta$. Given a stack $\cX$ with a point $\xi \in \cX(K)$, we define
\[
\bS(\cX,\xi)_n := \left\{ \begin{array}{c} \sigma : \ST_{R}^{n} \to \cX \text{ with} \\ \text{isomorphism } \sigma(\eta) \simeq \xi|_{K'} \end{array} \right\}.
\]
The \textbf{affine degeneration space, $\ADeg(\cX,\xi)$} is the topological space obtained by gluing together copies of the standard simplices $\Delta^n$ along rational linear maps, with one copy of $\Delta^n$ for each $\sigma \in \bS(\cX,\xi)_n$, in analogy with the geometric realization of a simplicial set. The combinatorial framework for studying $\bS(\cX,\xi)_\bullet$ is a slight modification of the framework of formal fans developed to define the degeneration space $\Deg(\cX,x)$ in \cite[\S 3]{halpernleistner2022structureinstabilitymodulitheory}.

\begin{example}
  When $\cG$ is a reductive group scheme over $R$ (or more generally a Bruhat-Tits group scheme), we expect $\ADeg(B\cG,\xi)$ to be naturally homeomorphic to the Bruhat-Tits building of $\cG_K$. Furthermore, when $\cX = \Spec A/ \cG$ for some affine $\cG$-scheme over $R$, we expect $\ADeg(\cX,\xi)$ to be a closed locally convex subset of this building.
\end{example}

A certain datum on $\cX$, called a \emph{quadratic norm} on graded points \cite[Def.~4.1.12]{halpernleistner2022structureinstabilitymodulitheory}, should canonically define a metric on $\ADeg(\cX,\xi)$. The proposed dictionary between algebro-geometric properties of $\cX$ and metric properties of $\ADeg(\cX,\xi)$ is the following:
\begin{itemize}
  \item The map $\Delta^n \to \ADeg(\cX,\xi)$ corresponding to a $\sigma \in \bS(\cX,\xi)_n$ is totally geodesic, for the usual flat metric on $\Delta^n$.\smallskip
  \item $\sS$-completeness of $\cX$ means that any two points of $\ADeg(\cX,\xi)$ are connected by a unique geodesic line segment.\smallskip
  \item A morphism $\gamma : \Theta_R \to \cX$ with an isomorphism $\gamma(1_K) \cong \xi$ corresponds to a geodesic ray in $\ADeg(\cX,\xi)$.\smallskip
  \item If $\cX$ is $\sS$-complete, $\Theta$-complete, and quasi-compact with affine diagonal, we expect $\ADeg(\cX,\xi)$ to be a Hadamard space, i.e., a complete CAT(0) space.\smallskip
  \item The ideal boundary (see \cite[II.8]{MR1744486} for definition) of $\ADeg(\cX,\xi)$ is naturally homeomorphic to the degeneration space $\Deg(\cX,\xi)$, so its rational points correspond to filtrations $f : \Theta_K \to \cX$ with an isomorphism $f(1)\cong \xi$, up to positive scaling.\smallskip
  \item $\Theta$-completeness means that any point of $\ADeg(\cX,\xi)$ is connected to any boundary point by a unique geodesic ray.\smallskip
\end{itemize}

\subsection{Application to equivariant degenerations}

One of the surprisingly subtle steps in establishing \Cref{thm:existence} was showing the following technical fact: if $V$ is a linear representation of a reductive group $G$, $\xi \in V(R)$, and $g\in G(K)$ is a finite-order element that fixes $\xi$, then there is a finite extension $R \subseteq R'$ with fraction field $K'$ and $h \in G(K')$ such that $h g h^{-1} \in G(R') \subseteq G(K')$ while still having $h \cdot \xi \in V(R') \subseteq V(K')$.

The proof in \cite{ahlh-existence} conceals the conceptual understanding that underlies it, which is based on the idea that $\ADeg(\cX,\xi)$ is a Hadamard space. In fact, if the picture outlined above holds, it implies the following stronger result:
\begin{proposition}\label{P:fixed_point}
    Suppose $\cX$ is a $\Theta$- and $\sS$-complete stack of finite type over a noetherian base\footnote{If the base is not characteristic $0$, then one additionally should assume it is locally reductive in the sense of \cite{ahlh-existence}.} and $\xi \in \cX(R)$. Then for any finite subgroup $\Gamma \subseteq \Aut_\cX(\xi_K)$, there is an element $\xi' \in \cX(R[\pi^{1/m}])$ for some $m>0$ and an isomorphism $\xi_{K[\pi^{1/m}]} \cong \xi'|_{K[\pi^{1/m}]}$ under which
    \[
    \Gamma \subseteq \Aut_{\cX}(\xi|_K) \subseteq \Aut_{\cX}(\xi|_{K[\pi^{1/m}]}) \cong \Aut_{\cX}(\xi'|_{K[\pi^{1/m}]})
    \]
    lies in the subgroup of $R[\pi^{1/m}]$-points of $\Aut_{\cX}(\xi')$, i.e., $\Gamma$ extends from the generic fiber to a group of automorphisms of $\xi'$.
\end{proposition}

The proof outline is the following: $\Aut_{\cX}(\xi|_K)$ acts naturally by isometries on $\ADeg(\cX,\xi)$, and therefore $\Gamma$ acts on $\ADeg(\cX,\xi)$. The point $\xi'$  corresponds to a rational fixed point for this $\Gamma$-action. The Bruhat-Tits fixed point theorem implies that any bounded group of isometries acting on a non-empty Hadamard space has a fixed point, hence $\ADeg(\cX,\xi)^\Gamma$ is non-empty. Because $\Gamma$ is finite and preserves the ``simplicial'' structure on $\ADeg(\cX,\xi)$, this fixed locus must also be a union of rational simplices, hence it must have a rational point.

\begin{example}
    If $X$ is a $K$-semistable $\bQ$-Fano variety with a finite group of symmetries $\Gamma$, then the proposition implies that it extends to a $\Gamma$-equivariant family of $K$-semistable $\bQ$-Fano varieties over $R$.
\end{example}

An interesting consequence of \Cref{P:fixed_point} is that one can ``trade-off'' between unramified and totally ramified extensions of $R$.

\begin{corollary}
    Suppose $\cX$ is a stack as in \Cref{P:fixed_point} and $\xi \in \cX(K)$ is such that after an unramified separable extension $K \subseteq K'$, $\xi$ extends to a point over the valuation ring $R' \subseteq K'$. Then $\xi|_{K[\pi^{1/m}]}$ extends to a point over $R[\pi^{1/m}]$ for some $m>0$.
\end{corollary}
\begin{proof}
    It suffices to assume $K \subseteq K'$ is Galois, with Galois group $G$. In this case $\ADeg(\cX,\xi)$ is the fixed locus for the action of $G$ on $\ADeg(\cX,\xi|_{K'})$, which is non-empty by the Bruhat-Tits fixed point theorem.
\end{proof}

This is a soft argument for a somewhat weaker version of the recent result in \cite{bejleri2025rootstackvaluativecriterion} that a good moduli space morphism $\cX \to X$ satisfies the existence part of the valuative criterion for properness if one only allows root stacks $\sqrt[n]{\Spec R} \to \Spec R$ as opposed to arbitrary finite extensions of $R$.

\subsection{Relationship with analytic geometry}

Berkovich analytic geometry provides another perspective on the question of classifying the extensions of a $K$-point of $\cX$ to an $R$-point. There are two constructions of $K$-analytic stacks associated to an algebraic stack $\cX$ over $R$: one can consider the analytification of the generic fiber $(\cX_K)^{an}$, or one can consider the ``generic fiber of the formal stack'' $(\cX^\wedge)_\eta$.\footnote{Both constructions are left Kan extensions of the corresponding functor from the category of finite type affine $R$-schemes to the category of finite type algebraic $R$-stacks with affine diagonal.} There is a canonical map $(\cX^\wedge)_\eta \to (\cX_K)^{an}$ which is an embedding of a compact analytic domain if $\cX$ is a separated scheme \cite{MR3330762}. In general, we define $\ADeg(\cX,\xi)^{an}$ to be the fiber of this map over the point in $(\cX_K)^{an}$ associated to $\xi \in \cX(K)$. There is a natural continuous morphism $\ADeg(\cX,\xi) \to \ADeg(\cX,\xi)^{an}$, and we expect this to be a homeomorphism under certain conditions.

From this perspective, the previous section implies there should be natural metrics on the fibers of the canonical map $(\cX^\wedge)_\eta \to (\cX_K)^{an}$, and that they are CAT(0) spaces (and perhaps Hadamard spaces) when $\cX$ has a separated good moduli space. Our evidence for this picture is the well-known relationship between $K$-analytic geometry and Bruhat-Tits buildings \cite{MR756316,MR1070709,remy2010bruhat,remy2012bruhat}. In the case of a reductive group scheme $G$ over $R$, the affine building associated to the generic fiber $\cB(G_K,K)$ naturally embeds in the non-archimedean homogeneous space $(G_K)^{an} / (G^\wedge)_\eta$, which is the fiber of $(BG)^\wedge_\eta \to (BG_K)^{an}$. 

\section{Moduli of singular curves and the MMP for $\bar{M}_g$} \label{sec:singular-curves}

While $\bar{M}_g$ is rational for small $g$, it is of general type for $g \ge 24$ by the
celebrated theorems of Harris--Mumford \cite{harris-mumford-mg} and Eisenbud--Harris
\cite{eisenbud-harris-mg}.  It follows from the equally celebrated results of
Birkar--Cascini--Hacon--McKernan \cite{bchm} that the canonical model $\bar{M}_g^{\can} := \Proj
\bigoplus_{d \ge 0} \H^0(\bar{\cM}_g, K_{\bar{\cM}_g}^{\tensor d})$ is a projective variety
birational to $\bar{M}_g$.  

\begin{question}[Hassett--Keel]
  Does $\bar{M}_g^{\can}$ have a modular interpretation?
\end{question}

In an effort to answer this question, Hassett proposed in \cite{hassett-genus2} to interpolate
between $\bar{M}_g$ and $\bar{M}_g^{\can}$ by considering the log canonical models
\begin{equation} \label{eqn:log-canonical-model}
  \bar{M}_{g}(\alpha):=
  \Proj \bigoplus_{d \geq 0} \H^0(\bar{\cM}_{g}, \lfloor d(K_{\bar{\cM}_{g}}+\alpha\delta) \rfloor),
\end{equation}
where $\delta$ is the boundary divisor. 
Note that
$\bar{M}_g = \bar{M}_g(1)$ (since $K + \delta$ is ample on $\bar{M}_g$) while
$\bar{M}_g^{\can} = \bar{M}_g(0)$.  This leads to the following refinement of the above question:

\begin{question} \label{ques:logMMP}
  For each $g$, determine the critical values  $1 > \alpha_1 > \alpha_2 > \cdots > \alpha_n \ge 0$
  where  $\bar{M}_g(\alpha)$ changes, provide a modular interpretation for each
  $\bar{M}_g(\alpha_i)$, and describe the geometry of the birational maps $\bar{M}_g(\alpha_i)
  \dashrightarrow \bar{M}_g(\alpha_{i+1})$ in terms of the moduli descriptions.
\end{question}

Except for phenomenon in low genus, the first three critical values are $9/11$, $7/10$, and $2/3$, and in particular independent of $g$. 
\begin{question} 
  Are the critical values $\alpha_i$ independent of $g$?
\end{question}

\subsection{Line bundles on $\bar{\cM}_g$} \label{subsec:line-bundles-on-mgbar}

There are several geometric ways to define line bundles on $\bar{\cM}_g$.  First, there is the
canonical line bundle $K$ and the boundary divisors $\delta_i$:
$$\begin{aligned}
K 
  & := K_{\bar{\cM}_g} = \det \Omega_{\bar{\cM}_g}\\
\delta_i 
  & :=  \{ [C] \in \bar{\cM}_g \mid 
      \text{$C$ is a nodal union of genus $i$ and genus $g-i$ curves } \} \\
\delta_0 
  & := \overline{ \{ [C] \in \bar{\cM}_g \mid \text{$C$ is irreducible with a node } \} }\\
\delta 
  & := \delta_0 + \delta_1 + \cdots + \delta_{\lfloor g/2 \rfloor} = 
   \{ [C] \in \bar{\cM}_g \mid \text{$C$ is singular} \} .
\end{aligned}$$
If $\pi \co \cU_g \to \bar{\cM}_g$ denotes the universal family,  then we can define the closely
related line bundles 
\begin{equation} \label{eqn:lambdak}
  \lambda_k  := \det \pi_* (\omega_{\cU_{g}/\bar{\cM}_{g}}^{\tensor k}) \qquad \text{and} \qquad
 \Lambda_k := \det \R \pi_* (\omega_{\cU_{g}/\bar{\cM}_{g}}^{\tensor k}), 
\end{equation}
where $\lambda := \lambda_1 = \Lambda_0 = \Lambda_1$ is called the \emph{Hodge line bundle} and
$\lambda_k = \Lambda_k$ for $k \ge 2$. The Knudsen--Mumford formula 
implies that there are line bundles $M_0$, $M_1$, and $M_2$ on $\bar{\cM}_g$ such that $\Lambda_k =
{k \choose 2} M_2 + k M_1 + M_0$.  In particular, this allows us to write $\lambda_k =
(6k^2-6k+1)\lambda - {k \choose 2} \delta$.  We define the \emph{Chow--Mumford} or \emph{CM line
bundle} as the leading coefficient
$L_{\CM} := M_2$.
Grothendieck--Riemann--Roch calculations imply that $K = 13 \lambda - 2 \delta$, $M_1 = 0$, and
$$L_{\CM} = M_2 = \pi_*(c_1(\omega_{\cU_{g}/\bar{\cM}_{g}})^2),$$
where the latter cohomology class is commonly referred to as the $\kappa$-class. 

It is a theorem of Harer (analytic) or Arbarello--Cornalba (algebraic) that $\Pic(\bar{\cM}_g)$ is
freely generated over $\bZ$ by $\lambda$ and the $\delta_i$ for $i \le \lfloor g/2 \rfloor$.  
Restricting to the
$\lambda\delta$-plane, we write a line bundle (up to a positive scalar multiple) as $s \lambda -
\delta$, where $s$ is referred to as 
the \emph{slope}.  This ray in the plane can be written equivalently as
 $K + \alpha \delta$, where $s = 13/(2-\alpha)$.  In the $\lambda\delta$-plane, the ample cone of $\bar{M}_g$ is known, but it is open problem to determine the effective cone; see \Cref{fig:effective-ample-cone}.  

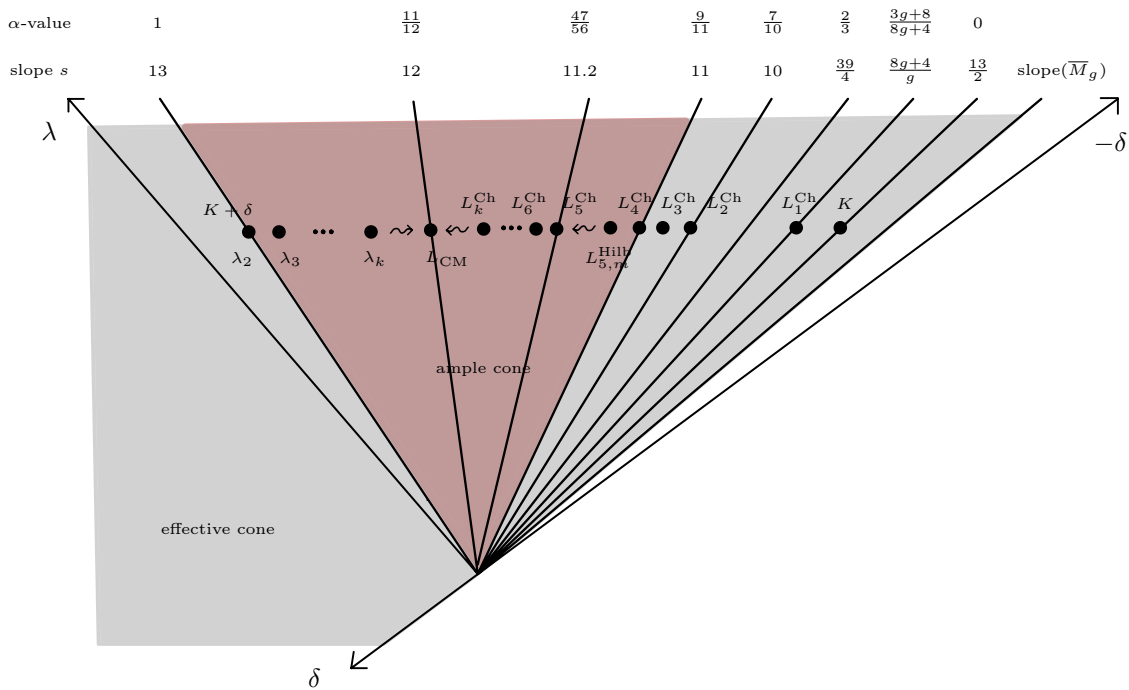
\begin{figure}[h!]
    \centering
    \labellist
      \pinlabel {\tiny{$\alpha$-value}}  at 66 775
      \pinlabel {\tiny{slope $s$}}  at 66 755
      
      \pinlabel {\tiny{13}} at 115 755
      \pinlabel {\tiny{12}} at 220 755
      \pinlabel {\tiny{11.2}} at 290 755
      \pinlabel {\tiny{11}} at 340 755
      \pinlabel {\tiny{10}} at 370 755
      \pinlabel {\tiny{$\frac{39}{4}$}} at 400 755
      \pinlabel {\tiny{$\frac{8g+4}{g}$}} at 427 755
      \pinlabel {\tiny{$\frac{13}{2}$}} at 455 755
      \pinlabel {\tiny{slope($\bar{M}_g$)}} at 490 755
      
      \pinlabel {\tiny{1}} at 115 775
      \pinlabel {\tiny{$\frac{11}{12}$}} at 220 775
      \pinlabel {\tiny{$\frac{47}{56}$}} at 290 775
      \pinlabel {\tiny{$\frac{9}{11}$}} at 340 775
      \pinlabel {\tiny{$\frac{7}{10}$}} at 370 775
      \pinlabel {\tiny{$\frac{2}{3}$}} at 400 775
      \pinlabel {\tiny{$\frac{3g+8}{8g+4}$}} at 427 775
      \pinlabel {\tiny{$0$}} at 455 775

      \pinlabel {$\lambda$}  at 70 730    
      \pinlabel {$\delta$}  at 180 503
      \pinlabel {$-\delta$}  at 510 725
      
      \pinlabel {\tiny{effective cone}}  at 140 565
      \pinlabel {\tiny{ample cone}}  at 250 631
      
      \pinlabel {\tiny{$\lambda_2$}} at 150 677
      \pinlabel {\tiny{$K+\delta$}} at 144 697        
      \pinlabel {\tiny{$\lambda_3$}} at 170 677
      \pinlabel {\tiny{$\lambda_k$}} at 205 677
      \pinlabel {\tiny{$L_{\CM}$}} at 235 677
      \pinlabel {\tiny{$L_{k}^{\Ch}$}} at 248 700
      \pinlabel {\tiny{$L_{6}^{\Ch}$}} at 269 700
      \pinlabel {\tiny{$L_{5}^{\Ch}$}} at 290 700
      \pinlabel {\tiny{$L_{5,m}^{\Hilb}$}} at 302 677
      \pinlabel {\tiny{$L_{4}^{\Ch}$}} at 313 700
      \pinlabel {\tiny{$L_{3}^{\Ch}$}} at 331 700
      \pinlabel {\tiny{$L_{2}^{\Ch}$}} at 350 700
      \pinlabel {\tiny{$L_{1}^{\Ch}$}} at 381 700
      \pinlabel {\tiny{$K$}} at 400 700
    \endlabellist
    \includegraphics[width=\textwidth]{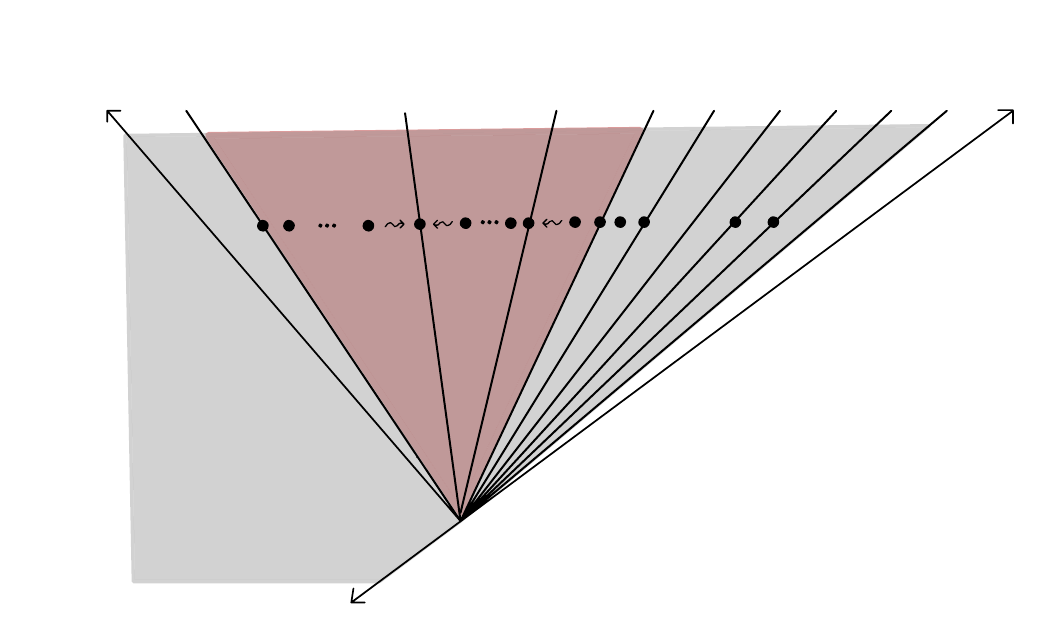}
    \caption{The effective and ample cone in the $\lambda\delta$-plane}
    \label{fig:effective-ample-cone}
\end{figure}

The GIT construction of $\bar{M}_g$ proceeds by choosing an integer $k$ and viewing a stable curve
$C$ under its $k$th-pluricanonical embedding $|\omega_C^{\tensor k}| \co C \into \bP^N$ after
choosing a basis of $\H^0(C, \omega_C^{\tensor k}) \cong \bC^{N+1}$. This embedded curve can be
viewed either as element of  $\Chow(\bP^N)$, which has a canonical projective embedding, or as
element of $\Hilb^P(\bP^N)$ (with $P$ the Hilbert polynomial), where a choice of $m \gg 0$ defines
an embedding into 
$\Gr\big(P(m), \Gamma(\bP^{N}, \oh(m))\big)$ via 
$[C \subseteq \bP^N] \mapsto 
  \big[\Gamma(\bP^{N}, \oh(m))
  \mapsonto \Gamma(C, \oh(m))\big]$
and into a projective space after composing with the Plücker embedding.  While
$\omega_C^{\tensor k}$ is very ample for $k \ge 3$, the GIT construction of $\bar{M}_g$ works
only for $k \ge 5$.  

The GIT construction yields ample line bundles $L^{\rm Ch}_{k}$ and $L^{\Hilb}_{k,m}$ for $k \ge 5$
and $m \gg 0$.\footnote{On the other hand, Kollár's construction \cite{kollar-projectivity} yields
the ampleness of $\lambda_k$ for $k \ge 6$.}   
For other values of $k$ and $m$, GIT constructs a projective variety birational to
$\bar{M}_g$.  The classes of these line bundles
are proportional to 
$$\begin{aligned}
	L^{\rm Ch}_{k} & \sim  \left\{
		 \begin{array}{ll}
			(8 + \frac{4}{g}) \lambda - \delta & \text{if } k = 1\\
		  (12-\frac{4}{k}) \lambda - \delta & \text{if } k > 1
		 \end{array} \right. \\
  L^{\Hilb}_{k,m} & \sim  \left\{
		 \begin{array}{ll}
			\big(\frac{2+(m-1)((8g+4)m-2g+2))}{gm}\big) \lambda - \delta & \text{if } k = 1\\
		  \big(\frac{12mk^2-4mk - 4k +2}{mk^2}\big)\lambda - \delta & \text{if } k > 1.
		 \end{array} \right. \\
\end{aligned}$$
For fixed $k$, $\lim_{m \to \infty} L^{\Hilb}_{k,m} \sim L^{\rm Ch}_k$, while 
$\lim_{k \to \infty} L^{\rm Ch}_k \sim L_{\CM}$; see \cite[Thm.~5.15]{mumford-git-stability},\
\cite[\S 1.5]{viehweg}, and \cite[\S2]{cornalba-harris}.

In the same paper \cite{mumford-git-stability} where Mumford constructed $\bar{M}_g$ using GIT on
the Chow scheme, he also observed that alternative GIT setups provide birational models of
$\bar{M}_g$ parameterizing curves with singularities that are possibly worse than nodal. 
Specifically, in genus $3$, he constructed a rational map 
$\bar{M}_3 \dashrightarrow \bP(\Sym^4(\bC^3))\gitq \SL_3$ 
contracting curves with \emph{elliptic tails} (i.e., a nodally attached genus $1$ curve) to
cuspidal curves, mapping \emph{elliptic bridges} (i.e., nodal union of two genus one curves along
two points) to tacnodal curves, and contracting the hyperelliptic locus to a point, which is
identified with the double conic. The MMP for low genus for $g \le 4$ is now complete:   $g=2$
\cite{hassett-genus2}, $g=3$ \cite{hyeon-lee-genus3}, and $g=4$ \cite{adlw-genus4}. 
 
 In arbitrary genus, Hassett and Hyeon construct the first divisorial contraction
 $\bar{M}_g \to \bar{M}_g(9/11)$ contracting the locus of elliptic tails
 \cite{hassett-hyeon-contraction} and  the first flip 
  $$\bar{M}_g(9/11) \to \bar{M}_g(7/10) \leftarrow  \bar{M}_g(7/10 - \epsilon)$$ 
for $\epsilon \ll 1$, flipping the locus of elliptic bridges to tacnodal curves
\cite{hassett-hyeon-flip}. For a detailed survey of GIT approaches to the log MMP, we recommend
\cite{fedorchuk-smyth-handbook}.

\subsection{An intrinsic approach} \label{sec:an-intrinsic-approach}
One non-GIT approach to \Cref{ques:logMMP} is to use heuristics to define a moduli stack
$\bar{\cM}_g(\alpha)$ of curves and prove that it admits a good moduli space 
$\bar{\cM}_g(\alpha) \to \bar{M}_g(\alpha)$, where $\bar{M}_g(\alpha)$ is the projective log 
canonical model defined in \eqref{eqn:log-canonical-model}.  Several heuristics can guide the search
for $\bar{\cM}_g(\alpha)$:
{ \setlength{\leftmargini}{10pt}
\begin{itemize}
  \item \emph{Openness:} $\bar{\cM}_g(\alpha)$ should be an \emph{open} substack of the
  stack of all curves.\footnote{It turns out that the stack of all curves has many open substacks
  with proper coarse or good moduli spaces, e.g., see \cite{smyth-compactifications}.  This is very
  much in contrast to the stack $\cCoh_{r,d}(C)$ of coherent
  sheaves on a curve, where  the only open substack with a proper good moduli space
  is $\cBun_{r,d}^{\ss}$ \cite{fh-w-z}.}
  \item \emph{Semistable reduction}: $\bar{\cM}_g(\alpha)$ should satisfy the existence part of the
  valuative criterion for properness. 
  \item \emph{Character theory:}  Since the line bundle $K+\alpha \delta$ should descend to an ample
  line bundle on $\bar{M}_g(\alpha)$, the action of $\Aut(C)$ on the fiber $(K+\alpha
  \delta)|_{[C]}$ should be trivial for each curve $[C] \in \bar{\cM}_g(\alpha)$.  Therefore, for
  particular singular curves $C$ with a $\bG_m$-action, one can determine precisely which values of
  $\alpha$ for which $C$ can appear in $\bar{\cM}_g(\alpha)$.  
  \item \emph{Intersection theory:} For special curves $T \subseteq \bar{\cM}_g$, one can compute
  precisely the value of $\alpha$ for which $(K+\alpha \delta) \cdot T = 0$.  At this critical
  value, $T$ must be contracted. 
  \item \emph{Inductive structure:} For each critical value $\alpha_c$, there should be open
  substacks 
  $\bar{\cM}_g(\alpha_c + \epsilon) 
    \into \bar{\cM}_g(\alpha_c) 
    \hookleftarrow \bar{\cM}_g(\alpha_c - \epsilon)$.
\end{itemize}}
\noindent 
See \cite{afs-predictions} for predictions of $\bar{\cM}_g(\alpha)$ using these heuristics.

Given a good candidate for $\bar{\cM}_g(\alpha_c)$ at a critical value $\alpha_c$, the next step is
to construct \'etale VGIT (variation of GIT) presentations.  Specifically, for each closed
point $[C] \in \bar{\cM}_g(\alpha_c)$, one constructs an \'etale neighborhood
$U/\Aut(C) \to \bar{\cM}_g(\alpha_c)$ of $[C]$, where $U$ is affine and the preimages of
$\bar{\cM}_g(\alpha_c \pm \epsilon)$ are identified with the semistable locus of $[U/\Aut(C)]$ with
respect to the character $\pm (K+\alpha_c \delta)|_{[C]}$.  The existence of a good moduli space for
$\bar{\cM}_g(\alpha_c)$ then follows inductively from the existence of good moduli spaces of
$\bar{\cM}_g(\alpha_c + \epsilon)$ and good moduli spaces of \emph{both} of the closed substacks 
$\bar{\cM}_g(\alpha_c) \minus \bar{\cM}_g(\alpha_c \pm \epsilon)$ \cite[Thm.~1.3]{second-flip2}. 
Establishing projectivity is (as usual) the most intricate step, requiring an analysis of the
intersection-theoretic properties of $K+\alpha \delta$.

This approach was developed  in  \cite{second-flip1,second-flip2,second-flip3} and applied to
construct the second flip  
$$\bar{M}_g(7/10-\epsilon)=\bar{M}_g(2/3+\epsilon) 
  \to \bar{M}_g(2/3) \leftarrow  \bar{M}_g(2/3 - \epsilon),$$ 
replacing genus 2 Weierstrass tails (i.e., curves with a genus $2$ subcurve meeting the rest
of the curve in a Weierstrass point) with ramphoid cusps.  It was also applied in \cite{ctv-mmp1}
and \cite{ctv-mmp2} to show for any subset $T \subseteq \{\irr, 2, \ldots, \lfloor g/2 \rfloor\}$,
there is an open substack $\bar{\cM}_g^T \subseteq \cM_g^{\all}$ admitting a projective good
moduli space $\bar{M}_g^T$ with birational maps 
$$\bar{M}_g \to \bar{M}_g(9/11) \xto{f_T} \bar{M}_g^T,$$
where $f_T$ contracts the $K$-negative face of the Mori cone $\bar{\NE}(\bar{M}_g(9/11))$
corresponding to $T$, namely non-separating elliptic bridges if $\irr \in T$ and genus $i$
separating elliptic bridges if $i \in T$.  When $T = \{\irr, 2, \ldots, \lfloor g/2 \rfloor\}$,
$\bar{M}_g^T = \bar{M}_g(7/10)$.  Moreover, they also identify $\bar{M}_g^T$ with the
projective model corresponding to certain adjoint line bundles 
$K_{\bar{M}_g} + \alpha_{\irr} \delta_{\irr} + \sum_{i} \alpha_i \delta_i$.
This provides a partial Shokuruv decomposition of
$\bar{\NE}(\bar{M}_g)$ in the boundary of the regions defining $\bar{M}_g$ and $\bar{M}_g(9/11)$.

\begin{problem}
  Work out the complete Shokuruv decomposition of $\bar{\NE}(\bar{M}_g(9/11))$ corresponding to
  birational models with tacnodal modular interpretations.
\end{problem}

\begin{problem}
  Work out the third flip at $19/29$ and extend the Shokuruv decomposition.
\end{problem}

While the above approach could potentially resolve these problems, the combinatorics of curve types
becomes increasingly intricate, especially for $\alpha \le 5/9$. To extend the log MMP of
$\bar{M}_g$ significantly further, a more intrinsic method based on the Beyond GIT framework may be
required.

\subsection{Beyond GIT for prestable curves}
We discuss the stack $\cM_g^{\all}$ of all curves in  \S \ref{subsec:stack-of-all-curves} and the
challenges of applying the Beyond GIT framework in \S \ref{subsec:beyond-git-curves}. For the
moment, we will explore $\Theta$-semistability conditions on the open  substack $\cM_g^{\pre}
\subseteq \cM_g^{\all}$ of prestable curves, i.e. connected nodal curves. 

The stack $\cM_g^{\pre}$, is rather poorly behaved.  It is not a global
quotient stack, does not  have quasi-affine diagonal, and does not contain a bounded 
open substack with a good moduli space other than $\bar{\cM}_g$; see \cite[\S
2.2]{alper-kresch}.  
Let $\cM_{g}^{\ss}$ denote the open substack of
semistable curves, i.e., connected nodal curves without nodally attached smooth rational tails.  There are open immersions
$$\bar{\cM}_g \subseteq \cM_g^{\ss} \subseteq \cM_g^{\pre},$$
where the complement of the first immersion is codimension $2$ and the second complement
is a divisor $\delta^{\pre} := \cM_{g}^{\pre} \setminus \cM_{g}^{\ss}$.
Let $\st \co \cM_{g}^{\pre} \to \bar{\cM}_g$ denote the stabilization morphism. 
In connection to \Cref{conj:stack-of-all-curves}(\ref{conj:stack-of-all-curves-3}), it is
not hard to see that the composition $\bar{\cM}_g^{\pre} \to \bar{\cM}_g \to \bar{M}_g$ 
is a categorical moduli space.

\begin{theorem}  Let $A \in \Pic(\bar{\cM}_g)$.
  \begin{enumerate}
    \item If $\cM = \cM_g^{\ss}$ and $\cL = \st^*(A)$, then $\cM^{\cL-\ss} = \cM_g^{\ss}$.
    \item If $\cM = \cM_g^{\pre}$ and $\cL = \st^*(A) + \beta \oh(\delta^{\pre})$, then
    $$ \cM^{\cL-\ss} = \left\{
      \begin{array}{ll}
        \cM_g^{\ss} & \text{if $\beta > 0$}\\
        \cM_g^{\pre} & \text{if $\beta = 0$}\\
        \emptyset & \text{if $\beta < 0$.}
      \end{array} \right. 
    $$
  \end{enumerate}
\end{theorem}
In particular, the semistable loci are open but do not admit good moduli spaces.

\section{Open problems} \label{sec:open-problems}

Moduli theory has always been half science and half art. The science deals with tools for solving classification problems in algebraic geometry, the earliest of which are geometric invariant theory and the Keel-Mori theorem for constructing moduli spaces. The artform of moduli theory involves discovering and deciding upon the ideal moduli problems to study. The tools typically guide this artform, for instance in deciding which stability condition to use or which ``singular'' objects to parameterize in order to obtain a proper moduli space. 

The recent developments in the theory of algebraic stacks and the beyond GIT framework vastly generalize the previous generation of tools, and they provide $\Theta$-stratifications in addition to moduli spaces. Our hope is that these tools will enable significant advancements in the art of moduli theory. In this section, we survey several topics that are ripe for artistic attention.

\subsection{A nearly universal moduli functor}

Many moduli problems arise as substacks of the stack of ``schemes with structure''
\begin{equation} \label{E:universal_problem}
  \cM_n(\cX) : T \mapsto \left\{
  \vcenter{\hbox{$
    \xymatrix@R=10pt{ P \ar[d]^-{\pi} \ar[r]^u & \cX \\ T \ar@/^/[u]^-{s_i}}  
  $}}
  \middle\vert \,
   \vcenter{\hbox{$\begin{aligned} &\pi\text{ is flat and proper\footnotemark} \\ &s_1, \ldots, s_n \text{ are sections\footnotemark} \end{aligned}$}}
\right\},
\end{equation}
\addtocounter{footnote}{-1}
\footnotetext{To have any hope of being algebraic, one should require $\pi$ to be projective \'etale-locally over $T$.}
\stepcounter{footnote}
\footnotetext{Instead of sections $s_i \co T \to P$, one could parameterize pairs $(P, \Delta)$, which could be realized by mapping into a stack of Weil divisors.}

\noindent where $\cX$ is an algebraic stack. This stack becomes algebraic if one imposes a condition that equips $\pi$ canonically with a relatively ample line bundle and if $\cX$ has affine automorphism groups and is sufficiently nice (quasi-separated and locally of finite type). 

\begin{example}
    When $\pi$ is required to be the base-change of $C \to \Spec \base$, i.e., $P \cong C_T$, and $\cX = B\GL_r$, this stack is $\coprod_d \cBun_{r,d}$.
\end{example}

\begin{example}
    When the geometric fibers of $\pi$ are nodal curves of genus $g$ and $\cX$ is a projective variety, the Kontsevich stability condition gives the stack 
    $\cK_{g,n}(\cX)$ of stable maps \cite{kontsevich}, which is used to define Gromov--Witten invariants. This has been extended to DM stacks \cite{abramovich-vistoli} 
    and tame algebraic stacks \cite{aov}.   
\end{example}

There is a long history of attempts to interpolate between these two examples by considering the subfunctor of $\cM_n$ where the geometric fibers of $\pi$ are nodal curves and the target is a quotient stack $X/G$.   The theory of stable quasi-maps identifies substacks that are proper DM stacks when $X$ is affine and the semistable locus of $X/G$ is DM \cite{MR3126932}, or more recently for any stack $\cX$ that admits a proper good moduli space and contains a stable point \cite{dilorenzo-inchiostro--stable-maps}. In order to study wall-crossing behavior for these moduli spaces, it would be nice to have a larger stack that contains many different stacks of stable quasi-maps as the open pieces of different $\Theta$-stratifications. A version of this was achieved for a fixed smooth curve in \cite{hl-fh-gauged-maps} and for arbitrary nodal curves with maps to $B\GL_N$ \cite{halpernleistner2025quantumoperationsringsymmetric}.

\begin{problem}\label{prob:universal_example}
Identify a suitable substack of $\cM_n(\cX)$, as large as possible, that admits a $\Theta$-stratification whose centers have proper good moduli spaces.
\end{problem}

\begin{example}
    An interesting special case is when $\cX = BG$ for a reductive group $G$. For the open substack of $\cM_n(BG)$ parameterizing $G$-bundles on nodal curves, this is the question of identifying a stack that can be used to define ``Gromov-Witten'' invariants for $BG$.
    There have been prominent constructions of open substacks of $\cM_n(BG)$ that admit proper flat good moduli spaces over $\bar{M}_g$ by Caporaso for $G= \bG_m$ \cite{caporaso--univeral-picard} and Nagaraj-Seshadri for $G=\GL_n$ \cite{MR1687729} (extending earlier work by Gieseker for $G=\GL_2$ \cite{MR739786}).
    
    In higher dimensions, the stack of $G$-bundles is poorly behaved even on a fixed smooth projective variety $X$. For $G=\GL_n$, the only approach we are aware of is to work with the stack of torsion free sheaves \cite{Simpson1994ModuliI}. There are also various ways to use torsion free sheaves to enlarge the stack of $G$-bundles so that it admits a semistable locus with a projective good moduli space, but there are arbitrary choices involved, and the points added no longer parameterize $G$-bundles \cite{MR2153406, MR2450609, MR2097106, gómez2024modulistackprincipalrhosheaves}. As an alternative, one may consider the substack of $\cM_n(BG)$ parameterizing $G$-bundles on various degenerations of $X$. This approach is motivated by results in the study of Yang-Mills flow on the space of connections on higher dimensional K\"ahler manifolds \cite{MR2097548}.
\end{example}

\subsection{The stack of all curves} \label{subsec:stack-of-all-curves}
Let $\cM_g^{\all}$ be the stack of all proper 
curves (possibly non-reduced, non-equidimensional, and non-connected) of genus $g$, and recall that 
we are working over an algebraically closed field $\base$ of characteristic $0$. 
There is very little literature about $\cM_g^{\all}$, except for its construction
as an algebraic stack \cite[Thm.~B.1]{hall-singular}, \cite[Prop.~3.3]{dJHS}, \tag{0D5A}, 
and \cite[\S 6.4]{alper-moduli}. 

\begin{conjecture} \label{conj:stack-of-all-curves} For each $g \ge 0$, 
\begin{enumerate}[(1)]
  \item \label{conj:stack-of-all-curves-1} 
    $\cM_g^{\all}$ satisfies Vakil's Murphy's law. 
  \item \label{conj:stack-of-all-curves-2} 
    $\Pic(\cM_g^{\all}) = 0$.
  \item \label{conj:stack-of-all-curves-3}
    The structure map $\cM_g^{\all} \to \Spec \base$ is a categorical moduli space.
\end{enumerate}
\end{conjecture}
\begin{remark}
  Recall that the genus of a curve $C$ is defined as $g = 1 - \chi(\oh_C)$.  Since taking the disjoint union
  of a genus $g$ curve and $g$ copies of $\bP^1$ yields a genus $0$ curve, this construction defines maps 
  $\cM_g^{\all} \to \cM_0^{\all}$.  In this way, if $\coprod_g \cM_g^{\all}$ satisfies Murphy's Law, 
  then so will $\cM_g^{\all}$ for each $g$. 

  Part \eqref{conj:stack-of-all-curves-3} would imply that $\cM_g^{\all}$ is connected, a fact that 
  follows from the connectedness of the Hilbert schemes $\Hilb^P(\bP^N)$. 
  More interestingly, it was recently shown that the substack of connected, reduced curves is connected
  \cite{bozlee-connectedness}.  On the other hand, examples of Mumford \cite{mumford-pathologies4}
  and Pinkham \cite[Thm.~1.11]{pinkham-deformations} show that $\cM_g^{\all}$ has irreducible
  components consisting entirely of singular curves.  
\end{remark}

\subsection{Beyond GIT for singular curves} \label{subsec:beyond-git-curves}

In an effort to apply the Beyond GIT machinery to singular curves, we first need to settle 
the delicate question of which moduli stack of curves to consider. \Cref{conj:stack-of-all-curves} suggests that the stack of all curves
$\cM_g^{\all}$ does not even have any line bundles with which to define semistability. So instead, we consider the stack 
$\cM_g^{\rm cp} \subseteq \cM_g^{\all}$ of \emph{canonically polarized} curves 
$C$, i.e., $C$ is Gorenstein with $\omega_C$ ample, or its substack $\cM_g^{\rm cp, red}$ of reduced 
canonically polarized curves.  
The stack $\cM_g^{{\rm cp}}$ admits the line bundles $\lambda_k$ defined in 
\eqref{eqn:lambdak}, which allows us to construct $K$ and $\delta$ as linear combinations of 
$\lambda_1$ and $\lambda_2$.  There is also the stack of polarized curves, parameterizing pairs
$(C, L)$ consisting of a curve $C$ and an ample line bundle $L$, which admits line bundles analogous 
to $\lambda_k$.  

Restricting our attention to $\cM := \cM_g^{{\rm cp, red}}$ and 
$\cL_{\alpha} := K + \alpha \delta$, we consider the $\Theta$-semistable locus
$\cM^{\cL_{\alpha}-\ss}$ with respect to $\cL_{\alpha}$. 
The hope is that $\cM^{\cL_{\alpha}-\ss}$ provides an answer to \Cref{ques:logMMP}, i.e., that
$\cM^{\cL_{\alpha}-\ss}$ is an open substack of $\cM_g^{\all}$ and admits a projective good moduli 
space identified
with the log canonical model $\bar{M}_g(\alpha)$.

When $\alpha = 11/12$, $\cL_{\alpha}$ is a rational multiple of the CM line bundle $L_{\CM}$.  
Results in K-stability imply that $\cM^{L_{\CM}-\ss} = \bar{\cM}_g$.  Moreover, 
for any $1 \ge \alpha \ge 9/11$, the generalized slope inequalities of \cite{ctv} imply that every 
stable curve is $\cL_{\alpha}$-semistable.  The converse may be true but at the moment we can 
only destabilize non-stable curves in the larger stack $\cM_g^{\rm cp}$ containing non-reduced curves.   
Interestingly, Xiao's slope inequality \cite{xiao} implies that smooth curves are 
$\cL_{\alpha}$-semistable for $\alpha > (3g+8)/(8g+4)$, which is precisely 
the critical value for when the hyperelliptic locus is expected to be contracted. 

While it is expected that these stacks $\cM_g^{\rm cp}$ and $\cM_g^{\rm cp, red}$ are
neither $\Theta$- nor $\sS$-complete, this has not been verified yet.  However, recently 
Gori, Modin, and Pernice proved that  the substack
consisting of curves with only $A_n$-singularities is not $\Theta$- nor $\sS$-complete.  
Therefore, the first approach outlined in \S \ref{sec:beyond-git-completeness} of finding 
$\Theta$- and $\sS$-complete enlargements will likely not succeed.  The second 
approach seems more promising by defining a numerical invariant using the CM line bundle $L_{\cM}$ and a norm on test configurations, e.g.,  the L$^p$ norms or the minimum norm \cite[Def. 2.34 and Lem. 2.37]{xu-book}.

\begin{problem}  
  For the CM line bundle  $L_{\CM}$ and a suitable norm on test configurations, does the stack 
  $\cM_g^{\rm cp}$ admit a $\Theta$-stratification whose semistable locus is $\bar{\cM}_g$?
\end{problem}
  
While this problem may not have a clean answer, there is a recent striking result
providing some evidence that a $\Theta$-stratification could exist:   
for a unibranch monomial curve $C$ with $\bG_m$-action defined by a semigroup 
$\Gamma = \{\gamma_0, \gamma_1, \ldots\}$ and for embeddings
$C \into \bP^N$ defined by the image of the map
$\bP^1 \to \bP^N$ taking $t \mapsto [1, t^{\gamma_1}, t^{\gamma_2}, \ldots]$,  the Kempf 
destabilizing worst 1-parameter subgroup of $[C \subseteq \bP^N] \in \Chow(\bP^N)$
stabilizes as $N \to \infty$ \cite{jackson-swinarski}.

\subsection{Semistability for higher dimensional varieties}
A crowning achievement of several innovative advances in the MMP and singularity theory over the last decade has been the construction of projective \emph{KSBA moduli spaces}, parameterizing 
higher dimensional schemes (and pairs) with ample canonical divisors and slc singularities \cite{MR4566297}. These are the higher-dimensional analogues of $\overline{M}_{g,n}$. For varieties (and pairs) with anti-ample canonical class, we have the theory of $K$-stability (\Cref{E:K-stability}). In between these two extremes, little is known in general. 
The field seems to be stuck at the first step --- what is a reasonable moduli functor to study for schemes equipped with a (non-canonical) ample line bundle?
Some progress has been made on the Calabi--Yau case despite serious boundedness issues \cite{abbdilw}, \cite{bl-calabi-yau}.

$K$-semistability can be formulated for any pair $(X,L)$, consisting of a reduced scheme $X$ and an ample line bundle $L$, in terms of test configurations, flat families of projective schemes with a relatively ample bundle $(\tilde{X},\tilde{L})$ over $\Theta_{\base}$, along with an isomorphism $(\tilde{X}_1,\tilde{L}_1) \cong (X,L^{\otimes n })$ for some $n>0$ \cite{paul2004algebraicanalytickstability, MR2274514}. The tensor power on $L$ means that this is not a $\Theta$-stability condition on the naive stack parameterizing flat families of projective schemes along with a relatively ample line bundle.

It is tempting to define some kind of stack in which $(X,L)$ is identified with $(X,L^n)$, but in addition to significant technical issues, the substack of semistable objects will not be bounded. For example, if $(E,L)$ is a polarized elliptic curve, then after scaling $L$ it can degenerate to a ``necklace'' of rational curves of arbitrary length, and all of these limit points will be $K$-semistable. One possible solution is the following:

\begin{definition}
    For a fixed Hilbert polynomial $P(t) = a_0 + \cdots + a_n t^n$, let $\cM_P$ denote the stack that sends a scheme $T$ to the groupoid of pairs $(\pi : X \to T, \omega \in \operatorname{Coh}(X))$, where
\begin{enumerate}
    \item $\pi$ is a flat proper morphism with reduced, equidimensional, and connected geometric fibers of dimension $n$, and
    \item $\omega$ is flat and fiber-wise pure of dimension $n$ over $T$
    \item Locally over $T$, $\exists m>0$  such that the hull $\omega^{[m]}$ is a $\pi$-ample invertible sheaf on $X$ that has Hilbert polynomial $a_0 + a_1 (mt) + \cdots + a_n (mt)^n$, and
    \item $\forall m \in \bZ$, $\omega^{[m]}$ is $T$-flat and commutes with base change.
\end{enumerate}
\end{definition}

The last two conditions are algebraic by \cite[Thm.~21]{kollár2009hullshusks}. Applied to the relative canonical sheaf $\omega = \pm \omega_{X/T}$ this is the ``Koll\'ar condition'' used to define families in the Fano and general type setting. But now we are allowing $\omega$ to be arbitrary.

Because we are not identifying $(X,\omega)$ with $(X,\omega^{[m]})$, we do not immediately encounter the boundedness issues above. On the other hand, a polarized variety $(X,L)$ admits more test configurations in $\cM_P$ than it does in the naive stack of polarized varieties. For instance, if $(\tilde{X},\tilde{L})$ is a test configuration over $\bA^1$ with $(\tilde{X}_1, \tilde{L}_1) \cong (X,L^m)$, $\tilde{L}$ might not have an $m^{th}$-root, so it does not necessarily come from a test configuration of $(X,L)$. But if the total space of $\tilde{X}$ is normal, we can choose a divisor $D = \sum_i a_i D_i$ on $X$ with $\oh_X(D) \cong L$, then defining $\tilde{D} = \sum_i a_i \overline{\bG_m \cdot D_i}$, we have that $\omega = \oh_{\tilde{X}}(\tilde{D})$ is an equivariant flat family of reflexive sheaves on $\tilde{X}$ with $\omega^{[m]} \cong \tilde{L} \langle n \rangle$ for some $n \in \bZ$. Therefore, it is possible that there are enough test configurations in $\cM_P$ to give the same notation of $K$-semistability as the one in \cite[Thm.~3.9]{MR2274514}, without introducing as many boundedness issues.

In fact, there is a natural refinement of the notion of $K$-semistability: Denote the projection $\pi : \cM_P \times X \to \cM_P$. Locally over $\cM_P$, we choose an $m$ such that $\omega^{[m]}$ is a $\pi$-ample, and consider the polynomial $\lambda(t) = b_0 + b_1 t + \cdots + b_{n+1} t^{n+1} \in \Pic(\cM_P)_\bQ \otimes \bQ[t]$ defined by
\[
\det(R\pi_\ast((\omega^{[m]})^q)) = b_0 + b_1 qm + b_2 (qm)^2 + \cdots + b_{n+1} (qm)^{n+1}.
\]
This the the higher dimensional analogue of \eqref{eqn:lambdak}. We can use these classes to define a \emph{polynomial semistability} condition: a test configuration $f : \Theta \to \cM_P$ is defined to be destabilizing if
\[
\wt f^\ast\left(t P(t) \frac{b_{n+1}}{a_n} - \lambda(t)\right) \in \bQ[t]
\]
has a negative leading order coefficient in $t$. We say that $(X,\omega) \in \cM_P$ is \emph{polynomial semistable} if there is no destabilizing filtration with $f(1) \cong (X,\omega)$, and \emph{polynomial stable} if the leading coefficient is strictly positive for any non-trivial test configuration. The polynomial sequence of line bundles $t P(t) \frac{b_{n+1}}{a_n} - \lambda(t)$ is chosen so that it is unchanged if one replaces $\omega \mapsto \pi^\ast(L) \otimes \omega$ for any line bundle $L$ on $T$, and its leading term is the CM line bundle \cite{MR1312688}. This implies that $K$-stable $\Rightarrow$ polynomial stable $\Rightarrow$ polynomial semistable $\Rightarrow$ $K$-semistable.

\begin{question}
    Does the polynomial semistable locus of $\cM_P$ admit a separated, finite-type good moduli space, i.e., is it open, bounded, and $\Theta$- and $\sS$-complete? Is it the open piece of a $\Theta$-stratification of $\cM_P$?
\end{question}

\begin{remark}
    If we pull the polynomial sequence of $\bQ$-line bundles $t P(t) \frac{b_{n+1}}{a_n} - \lambda(t)$ back along the morphism $\Fanos_{n,V} \to \coprod_P \cM_P$, the leading term is the CM line bundle, used to define $K$-semistability of Fano varieties. Therefore the notion of polynomial semistability \emph{refines} $K$-semistability. Furthermore, the CM line bundle must have weight $0$ with respect to any test configuration in $\Fanos^{K-\rm{ss}}_{n,V}$, so $\Fanos_{n,V}^{\rm{poly-ss}} \subseteq \Fanos^{K-\rm{ss}}_{n,V}$ is the semistable locus for the sub-leading-order piece of $t P(t) \frac{b_{n+1}}{a_n} - \lambda(t)$. It follows by ``variation of good moduli space'' \cite[\S 5.6]{halpernleistner2022structureinstabilitymodulitheory} that $\Fanos_{n,V}^{\rm{poly-ss}}$ admits a good moduli space $M^{\rm{poly}}_{n,V}$ that is projective over $M^{\rm{ps}}_{n,V}$, and $t P(t) \frac{b_{n+1}}{a_n} - \lambda(t)$ is relatively ample for $t \gg 0$. The moduli space $M^{\rm{poly}}_{n,V}$ can separate $K$-semistable varieties that are contracted to the same point in $M^{\rm{ps}}_{n,V}$.
  \end{remark}

\subsection{Moduli of objects in linear categories}

Besides template \eqref{E:universal_problem}, another flavor of commonly-studied moduli problems concerns moduli of objects in an additive category. A version of particular interest is the moduli of semistable objects for a Bridgeland stability condition on a triangulated dg-category. The most general existence result in this setting is \cite[Thm.~2.3.1]{halpernleistner2025spaceaugmentedstabilityconditions}. Over a field $\base$ of characteristic $0$, if $\cC$ is a smooth and proper dg-category and $\sigma$ is a stability condition on $\cC$, then the moduli of semistable objects of any class in $K_0(\cC)$ admits a proper good moduli space if for any $G \in \cC$, there is a constant $c$ such that
\begin{equation}\label{E:mass_hom_bound}
    \dim \Hom(G,E) \leq c |Z(E)|
\end{equation}
for all semistable $E \in \cC$, where $Z(-)$ is the central charge. This builds upon a rich literature of constructions of moduli spaces of object in abelian categories 
\cite{toda}, \cite{toda-piyaratne}, \cite{bayer-macri-families}, 
\cite[\S 7]{ahlh-existence}, and \cite{lampetti}. In \cite{halpernleistner2025spaceaugmentedstabilityconditions}, it is conjectured that the ``mass-Hom'' inequality \eqref{E:mass_hom_bound} holds for all stability conditions on any smooth and proper category.

\subsection{Projectivity}

Besides the Nakai--Moishezon criterion, the only general tool we are aware of to establish projectivity of a good moduli space is Kollár's Ampleness Criterion \cite{kollar-projectivity}, which is an adaptation of an earlier criterion due to Viehweg \cite{viehweg-weak-positivity}.  Kollár's Ampleness Criterion has been applied, for instance, to prove both the projectivity of the KSBA moduli space and the K-moduli space.  However, this criterion doesn't apply to many other moduli spaces, e.g., Bridgeland semistable objects in a triangulated dg-category, resulting in a frustrating gap in the intrinsic approach to moduli theory. 

Faltings introduced an approach in \cite{faltings-stable-G-bundles-and-connections} to show ampleness of determinantal line bundles by explicitly constructing global sections, which Faltings applied to semistable vector bundles (and Higgs bundles) on a smooth curve. This was recently extended to reconstruct the Uhlenbeck compactification, parameterizing slope semistable bundles on a smooth surface, as a projective variety by reinterpreting it as the Bridgeland moduli space on the vertical wall \cite{tajakka}.  However, 
it appears difficult to extend this approach to other Bridgeland moduli spaces.  

Intersection-theoretic techniques can also be used to prove projectivity of moduli spaces,  such as
Cornalba and Harris's approach to prove that $\bar{M}_g$ is projective  \cite{cornalba-harris,cornalba-projectivity}.  This approach was extended in \cite{second-flip3} to show the 
projectivity of the good moduli spaces of the moduli stacks $\bar{\cM}_g(\alpha)$ of singular curves
discussed in \S\ref{sec:an-intrinsic-approach}, which in addition leveraged
the finite generation results of \cite{bchm} to verify semiampleness, but this method 
unfortunately only seems to apply to a space that is a birational model of a variety already known to be projective.

\subsection{Applications of combinatorial structures}

The degeneration space $\Deg(\cX,x)$ is the geometric realization of a combinatorial object called the \emph{degeneration fan}, a structure analogous to a semisimplicial set. If $X$ is a normal toric variety and $x \in X(\base)$ is generic, then one can recover the fan of the toric variety from the degeneration fan. When $\cX = BG$, one can essentially recover the spherical building of $G$. This raises the following:
\begin{question}
    If $\cX$ is a normal algebraic stack over $\base$ and $x \in \cX(\base)$ is a dense point, to what extent does the degeneration fan of $(\cX,x)$ encode the geometry of $\cX$? Can one describe vector bundles on $\cX$ or the cohomology of $\cX$ from combinatorial data on this fan?
\end{question}

An analogous question applies to the affine degeneration space and its underlying combinatorial object. Another application of the affine degeneration space would be to studying the following:
\begin{question} \label{Q:canonical}
If $\cX$ has a separated good moduli space and $\xi \in \cX(K)$ admits an extension over $\Spec R$, to what extent is there a \emph{canonical} extension over $\Spec(R[\pi^{1/m}])$ for some $m>0$?
\end{question}
The meaning of \emph{canonical} is somewhat open-ended.
When $\xi$ is polystable, one can use Kirwan's canonical blow-up procedure \cite{MR799252} to find a canonical extension. But it would be interesting to have a construction that uses the geometry of $\ADeg(\cX,\xi)$, and that works for semistable $\xi$ as well. Pursuing an analogy between the Hadamard space $\ADeg(\cX,\xi)$ and the Riemannian symmetric space $G/K$, where $G$ is a complex reductive group and $K$ a maximal compact subgroup, it would be interesting to make a connection with well-known differential-geometric methods in the theory of symplectic reduction (see for instance \cite{MR4387640}). One may hope to discover a canonical point of $\ADeg(\cX,\xi)$ as the minimizer of some natural convex ``energy'' functional, in analogy with \cite{2014arXiv1411.6786M, MR3722572}.

\subsection{Local structure theorems in positive characteristic}
 While the Local Structure Theorems hold over any characteristic, their
 applicability is very restrictive in positive characteristic due to the
 \emph{linearly reductive} hypothesis: every smooth linearly reductive group in
 characteristic $p$ is an extension of a finite group whose order is prime to
 $p$ by a torus.  For instance, $\GL_n$ is {\it not} a linearly reductive group
 scheme in positive characteristic. The class of \emph{reductive} groups, on the
 other hand, is much broader including classical groups such as $\GL_n$.  In particular, 
  the Local Structure Theorems are unknown for moduli stacks of sheaves
 or complexes in positive characteristic. 
 
While it is too optimistic to expect the existence of an étale neighborhood which is 
a quotient $\Spec A/G_x$ by the stabilizer, it is reasonable to hope that Local Structure Theorem II (\Cref{thm:local-structure}) still holds. 

\begin{conjecture} \label{conj:local-structure-charp} 
  Let $\cX$ be a quasi-separated algebraic stack, locally of finite type over a
  noetherian scheme, with affine automorphism groups.  If $x \in \cX$ is a
  closed point with reductive stabilizer, then there exists an étale morphism 
	$$(\Spec A/\GL_n, w) \to (\cX,x)$$ 
  inducing an isomorphism of automorphism groups at $w$.
\end{conjecture}

Resolving this conjecture would allow us to remove the characteristic zero
hypothesis from the Existence Theorem (\Cref{thm:existence}) below after
replacing the notion of a good moduli space with an \emph{adequate moduli space}
\cite{alper-adequate}.  On the other hand, there is not much evidence for this conjecture
except for the cases of linearly reductive or
finite discrete automorphisms.  
It is not even known for a quotient stack $X/\GL_n$ of a quasi-projective variety
in positive characteristic. In many ways, the theory is at an analogous stage to
GIT in the 1960s, which at first only allowed for constructions of moduli spaces
in characteristic 0.  It was only later, after Haboush resolved Mumford's conjecture
\cite{haboush}, that GIT could be widely applied in positive characteristic.  Perhaps
unsurprisingly, one of the most promising strategies for
\Cref{conj:local-structure-charp}  is to apply vanishing results in
characteristic $p$ representation theory as in Haboush's proof, and specifically
to exploit the magical cohomological properties of Steinberg representations.

\bibliographystyle{dary}
\bibliography{icm-refs}

\end{document}